\documentclass[pdflatex,sn-mathphys-num]{sn-jnl}

\usepackage{graphicx}%
\usepackage{multirow}%
\usepackage{amsmath,amssymb,amsfonts}%
\usepackage{amsthm}%
\usepackage{multirow}%
\usepackage[title]{appendix}%
\usepackage{xcolor}%
\usepackage{textcomp}%
\usepackage{manyfoot}%
\usepackage{booktabs}%
\usepackage{algorithm}%
\usepackage{algorithmicx}%
\usepackage{algpseudocode}%
\usepackage{listings}%

\theoremstyle{thmstyleone}%
\newtheorem{theorem}{Theorem}
\newtheorem{proposition}[theorem]{Proposition}%

\newtheorem{example}{Example}%
\newtheorem{definition}{Definition}%
\newtheorem{assumption}{Assumption}
\newtheorem{alg}{Algorithm}
\newtheorem{lemma}{Lemma}
\newtheorem{condition}{Condition}
\begin{document}

\title[Article Title]{Distributed accelerated proximal gradient methods for multi-agent constrained optimization problem}


\author[1]{\fnm{Anteneh Getachew} \sur{Gebrie}}

\affil[1]{\orgdiv{Department of Mathematics}, \orgname{Debre Berhan University},   \state{Debre Berhan}, \country{Ethiopia}}


\abstract{The purpose of this paper is to introduce two new classes of accelerated distributed proximal gradient algorithms for multi-agent constrained optimization
problems; given as minimization of a function decomposed as a sum of $M$ number of smooth and $M$ number of non-smooth functions over the common fixed points of $M$ number of nonlinear mappings. Exploiting the special properties of the cost the component function of the objective function and the nonlinear mapping of the constraint problem of each agent, a new inertial accelerated incremental and parallel computing distributed algorithms will be presented based on the combinations of computations of proximal, gradient, and Halpern methods. Some numerical experiments and comparisons are given to illustrate our results.}




\maketitle

\section{Introduction}  
In last few decades, the theory and practice of distributed methods for optimization problems is increasingly popular and has shown significant advances. A broad range of problems arising in several fields of applications, for example, in sensor networks \cite{1a}, economic dispatch \cite{2a},  smart grid \cite{3a}, resource allocation \cite{4a}, and machine learning \cite{5a}, can be posed in the framework distributed mannered multi-agent constrained or unconstrained optimization problems. Multi-agent constrained optimization problem in general takes the form of 
\begin{subequations}\label{eq:prob12}
	\begin{eqnarray} 
	&&\mbox{minimize } \psi(x)=\sum\limits_{i=1}^{M}\mathcal{G}_{i}(x)\label{eq:prob1a}\\&&\hspace{1mm}\mbox{subject to } x\in\bigcap_{i=1}^{M} \{y\in H:y\text{ solves } \mathcal{P}_{i}\},\label{eq:prob1b}
	\end{eqnarray}
\end{subequations}
where  $\mathcal{G}_{i}:H\rightarrow\mathbb{R}\cup \{+\infty\}$ ($i\in I$ ) is a function,  $\mathcal{P}_{i}$ ($i\in I$ ) a problem installed in space $H$ ($\mathcal{P}_{i}$ is a constraint problem (subproblem) of the problem (\ref{eq:prob12}) installed in space $H$), $H$ is a real Hilbert space, and $I=\{1,\ldots,M\}$.

The main goal of studying the distributed methods for multi-agent constrained optimization problem of the form (\ref{eq:prob12}) is to cooperatively solve the problem using private information available to agent $i$, i.e., the local objective function $\mathcal{G}_{i}$ and the subproblem $\mathcal{P}_{i}$ determined by agent $i$. Distributed methods for solving the multi-agent constrained optimization problem (\ref{eq:prob12}) involve $M$ number of agents where each agent $i$ maintains an estimate of the solution of (\ref{eq:prob12}) and updates this estimate iteratively using his own private information and information exchanged with neighbors over the network, see for example \cite{6a,7a,8a,9a}. In fact all cost component functions $\mathcal{G}_{i}$ may not have the same property and the properties of each cost component function may not be well suited for the same iteration method, so it makes sense to consider combinations of two or more iteration methods, for example gradient or subgradient and proximal methods, see for example \cite{7,8,20,21,7a,8a,9a}. 

Proximal \cite{46} and gradient method \cite{29} have a long history with many contributors in the study of optimization problems in general, and are prized for their applicability to a broad range of nonsmooth and smooth convex optimizations and for their good theoretical convergence properties. The most famous and the well known proximal gradient method is the forward-backward splitting (FBS) \cite{33,34} for the unconstrained optimization problem with objective function given as a sum of two functions (sum of one nonsmooth and one smooth function).  Many engineers and mathematicians use the gradient method to solve the large-scale optimization problems for its simplicity and low memory requirement \cite{23,24,29}. Moreover, in many practical applications accelerating the convergence the sequence generated by iterative method is required, see, for example, \cite{47,35,36}. The momentum acceleration method of Polyak’s heavy ball \cite{22} is the widely used acceleration technique. Polyak \cite{22} firstly introduced the momentum technique to accelerate the trajectories and speed up convergence in numerical computations for the two-order time dynamical system in the context of minimization of a smooth convex function, see \cite{8,9} for discussions in this direction.

Therefore, for both actual implementation and to create a firm foundation for the theory of
optimization problem, a most crucial thing would be to devise a more comprehensive multi-agent constrained optimization problem of the form (\ref{eq:prob12}), for example, some cost component functions of $\psi$ are nonsmooth and the remaining are smooth and the constraint problem $\mathcal{P}_{i}$ generalizes several previously considered constraint problems, and to investigate a unified accelerated solution method so that the method handles several previously-discussed optimization problems, and its implementation and convergence is significantly better, both theoretically and practically. This is the important question we address in this paper. To be precise, we consider multi-agent constrained optimization of the form (\ref{eq:prob12}), called, \textit{multi-agent fixed point set constrained optimization problem} (in short, MAFPSCOP), where for each agent $i$ the component function $\mathcal{G}_{i}$ is again a composite of two functions given by
\begin{eqnarray} \label{eq:prob33}
\mathcal{G}_{i}(x)=f_{i}(x)+h_{i}(x)
\end{eqnarray}
and the constraint problem $\mathcal{P}_{i}$ is a fixed point problem given by
\begin{eqnarray} \label{eq:prob44}
\mathcal{P}_{i}:\hspace{4mm}\text{find }  y\in H \text{ such that }T_{i}(y)=y,
\end{eqnarray}
where $f_{i}:H\rightarrow\mathbb{R}$ $(i\in I)$ is convex and nonsmooth
function, $h_{i}:H\rightarrow\mathbb{R}$ $(i\in I)$ is convex and
smooth function, and $T_{i}:H\rightarrow H$ $(i\in I)$
is a nonlinear mapping. Note that here, $\{y\in H:y\text{ solves } \mathcal{P}_{i}\}$ is the set of fixed points of $T_{i}$, denoted by $FixT_{i}$, i.e., $\{y\in H:y\text{ solves } \mathcal{P}_{i}\}=\{y\in H: T_{i}(y)=y\}=FixT_{i}$. The problem MAFPSCOP, i.e., the multi-agent constrained optimization problem (\ref{eq:prob12}) where $\mathcal{G}_{i}$ given by (\ref{eq:prob33}) and $\mathcal{P}_{i}$ given by (\ref{eq:prob44}), generalizes several optimization problems considered in literature and is abundant in many theoretical and practical areas, see for example, \cite{1,2,4,5,6,7,8,9,16,38} and references therein.

Motivated by the ideas in \cite{5,6} and inspired by the theoretical and practical applications of constrained optimization problem, we investigate efficient distributed algorithms for solving MAFPSCOP. Two new class of accelerated distributed proximal gradient methods for solving MAFPSCOP are developed using incremental and parallel computation approach \cite{37} incorporating the proximal \cite{46}, gradient \cite{23},  Halpern \cite{11a}, and heavy ball momentum \cite{22} methods. Our results, not only improve and generalize the several corresponding results, for example \cite{5,6,21,22,38}, but also provide a unified framework for studying problems of the form (\ref{eq:prob12}). 

The rest of the paper is organized as follows. In Section \ref{Sec2}, we provide some basic definitions and propositions from nonlinear and convex analysis. In Section \ref{Sec3}, we present our main results, i.e., we present and analysis the two algorithms proposed for MAFPSCOP. In Section \ref{Sec4}, we perform several numerical experiments to illustrate the efficiency of our algorithms in comparison with others. Finally, at the end of the paper we include an appendix devoted to the proofs of all lemmas we stated in the main result section (Section \ref{Sec3}).
\section{Required basics from convex analysis}\label{Sec2}
In this we recall known results from convex analysis. Unless otherwise stated, $H$ is a real Hilbert space with inner product $\langle .,. \rangle$ and its induced norm is $\|.\|$.\vspace{2mm}
\begin{definition} The mapping mapping $T:H\rightarrow H$ is called: 
	\begin{description}		
		\item[(a).]  $\gamma$-Lipschitz ($\gamma>0$) if
		$$\|Tx-Ty\|\leq \gamma\|x-y\|\mbox{ for all } x,y\in H.$$	
		\item[(b).]  $\gamma$-contraction mapping if $T$ is $\gamma$-Lipschitz with $\gamma\in (0,1)$. 	
		\item[(c).] nonexpansive mapping if $T$ is $1$-Lipschitz.
		\item[(d).]  firmly nonexpansive if\\
		$$\|Tx-Ty\|^{2}\leq \|x-y\|^{2}-||(I-T)x-(I-T)y\|^{2}, \forall x,y\in H,$$
		which is equivalent to	
		$\|Tx-Ty\|^{2}\leq \langle Tx-Ty,x-y\rangle,\hspace{1mm}\forall x,y\in H.$
	\end{description}	
\end{definition}
If $T$ is firmly nonexpansive, $I-T$ is also firmly nonexpansive. The class of firmly nonexpansive mappings include many types of nonlinear operators arising in convex optimization, see, for instance, \cite{40}.\vspace{2mm}

\begin{definition} Let $X$ be a Banach space. $X$ satisfies Opial's condition if for each $x$ in $X$ and each sequence $\{x_{n}\}$ weakly convergent to $x$ the inequality
	$$\liminf_{n\rightarrow \infty}\|x_{n}-x\|<\liminf_{n\rightarrow \infty}\|x_{n}-y\|$$
	holds for each $y\in X$ with $y\neq x$.\vspace{2mm}
\end{definition}	

It is well known that all Hilbert spaces satisfy Opial's condition \cite{111ABC}.

A set-valued mapping $V:H\rightarrow 2^{H}$ is called \emph{monotone} if, for all $x,y\in H$, $z\in Vx$ and $w\in Vy$ imply
$\langle x-y,z-w\rangle\geq 0.$

Let  $g:H\rightarrow\mathbb{R}\cup\{+\infty\}$ convex function. 
The domain of $f$ (dom$f$) is defined by $dom g=\{x\in H:g(x)<\infty\}$. 
\begin{description}
	\item[(i).]	The \emph{subdifferential} of $g$ at $x\in H$, denoted by $\partial g(x)$, is given by
	$\partial g(x)=\{y\in H:g(z)\geq g(x)+\langle y,z-x \rangle, \forall z\in H\}.
	$ $g$ is \emph{subdifferentiable} at $x$ if $\partial g(x)\neq\emptyset$.
	Since $g$ is convex, $\bar{x}\in H$ minimizes $g$ if and only if $0\in \partial g(\bar{x})$. 
	
	\item[(ii).] If $g$ is proper, lower semicontinuous, and convex function, then the \emph{proximal 
		operator} of $g$ with scaling parameter $\lambda>0$ is a mapping
	$\mbox{prox}_{\lambda g}:H\rightarrow H$ given by
	$\mbox{prox}_{\lambda g}(x)=\arg\min\{g(y)+\frac{1}{2\lambda}\|x-y\|^{2}:y\in H\}.$
	The notion of the proximal mapping was introduced by Moreau \cite{46}.

	\item[(iii).] $g$  is said to be \emph{Fr\'echet differentiable} at $x\in H$ if there exists a
	unique $g_{x}\in H$ such that 
	$\lim_{0\neq y\rightarrow 0}\frac{g(x+y)-g(x)-\langle y,g_{x}\rangle}{\|y\|}=0.$
	For $x\in H$, if such $g_{x}\in H$ exists, then $g_{x}$ is called the \emph{Fr\'echet gradient} (\emph{gradient}) of $g$ at $x$, denoted by $\nabla g(x)$, i.e., $\nabla g(x)=g_{x}$. We say that $g$ is \emph{Fr\'echet differentiable} if $g$ is \emph{Fr\'echet differentiable} on $H$, i.e., $\nabla g(x)=g_{x}$ exists for all $x\in H$.
\end{description}

\begin{proposition}\cite{40} 
	If $h:H\rightarrow\mathbb{R}\cup\{+\infty\}$ is proper, convex and Fr\'echet differentiable function, then for $x\in dom h$ we have $\partial h(x)=\{\nabla h(x)\}$.\vspace{2mm}
\end{proposition} 

\begin{proposition}\label{P4}\cite{38}
	Let $h:H\rightarrow\mathbb{R}$ be convex and
	Fr\'echet differentiable, and let $\nabla h:H\rightarrow H$ be $\frac{1}{L}$-Lipschitz continuous. For $\lambda\in[0,2L]$, we define $S_{\lambda}:H\rightarrow H$ for all $x\in H$ by $S_{\lambda}(x)=x-\lambda \nabla h(x)$. Then $S_{\lambda}$ is
	nonexpansive, i.e., 
	$\|
	S_{\lambda}(x)-S_{\lambda}(y)\|\leq\|x-y\|$ all $x,y\in H$.\vspace{2mm}
\end{proposition}
See more about the properties of Fr\'echet differentiable convex function $h:H\rightarrow\mathbb{R}$ and its gradient operator $\nabla h$ in Bauschke and Combettes \cite{40}. \vspace{2mm}

\begin{proposition}\label{P2} \cite{40} Let $f:H\rightarrow\mathbb{R}\cup\{+\infty\}$ be proper, lower semicontinuous, and convex function, and let $\lambda>0$. Then, the following hold:
	\begin{description}
		\item[(a).]	The subdifferential $\partial f$ of $f$ is a monotone operator.
		\item[(b).] For $x\in H$, we have
		$\mbox{prox}_{\lambda f}(x)=y \mbox{ if and only if } x-y\in\lambda\partial f(y).$
		\item[(c).]
		The proximal mapping is a
		firmly nonexpansive mapping.
		\item[(d).]
		If $f$ is continuous at $x\in\mbox{dom}f$, $\partial f(x)$ is nonempty. Moreover, $\delta>0$ exists such that $\partial f(B(x,\delta))$ is bounded, where $B(x,\delta)$
		stands for a closed ball with center $x$ and radius $\delta$.\vspace{2mm}
	\end{description}
\end{proposition}
\begin{proposition}\label{P3}\cite{40}
	When $f:H\rightarrow\mathbb{R}$ is convex, $f$ is weakly lower semi-continuous if and only if $f$ is lower semicontinuous.\vspace{2mm}
\end{proposition}

\begin{proposition}\cite{10a}\label{P5}
	Let $\{c_{n}\}$ and $\{\beta_{n}\}$ be a sequence of real numbers such that $c_{n}\geq 0$ and 
	$c_{n+1}\leq (1-\alpha_{n})c_{n}+\alpha_{n}\beta_{n}$ for all $n\geq 1$,
	where  $0<\alpha_{n}\leq 1$. If $\sum\alpha_{n}=\infty$ and   $\limsup_{n\rightarrow \infty}\beta_{n}\leq 0$, then $c_{n}\rightarrow 0$ as $n\rightarrow \infty$.
\end{proposition}
\section{Proposed methods}\label{Sec3}
In this section, under the following assumptions and parameter restrictions,  we propose and analysis two new accelerated algorithms for MAFPSCOP. 
\begin{assumption}\label{assumption1} In MAFPSCOP, the following basic assumptions on $f_{i}$, $h_{i}$ and $T_{i}$ are made throughout the paper:
	\begin{description}
		\item[\textbf{(A1)}]  $f_{i}:H\rightarrow\mathbb{R}$ $(i\in I)$ is continuous and convex with
		$dom(f_{i})=H$ and $\mbox{prox}_{\lambda f_{i}}$ can be efficiently computed; 
		\item[\textbf{(A2)}] $h_{i}:H\rightarrow\mathbb{R}$ $(i\in I)$ is convex and
		Fr\'echet differentiable with $dom(h_{i})=H$, and $\nabla h_{i}:H\rightarrow H$ $(i\in I)$ is $(1/L_{i})$-Lipschitz continuous; 
		\item[\textbf{(A3)}] $T_{i}:H\rightarrow H$ $(i\in I)$ is firmly nonexpansive mapping;	
		\item[\textbf{(A4)}]  $S=\bigcap_{i=1}^{M}FixT_{i}\neq\emptyset$ and $\Omega=\big\{\hat{x}\in S: \psi (\hat{x})=\min_{x\in S}\psi (x) \big\}\neq\emptyset.$
	\end{description}
\end{assumption}

\begin{condition}\label{condition1}
	Let $\{\theta_{n}\}$, $\{\lambda_{n}\}$, $\{\beta_{n}\}$ and $\{\alpha_{n}\}$ be decreasing real sequences converging to 0 where $\theta_{n}\in [0,1)$,  $\lambda_{n}\in (0,2\min_{i\in I}L_{i}]$, $\beta_{n}\in (0,1]$, $\alpha_{n}\in (0,1]$, 
	and satisfy the following conditions.
	\begin{description}
		\item[\textbf{(C1)}]  $\sum\limits_{n=1}^{\infty}\alpha_{n}=\infty$;\hspace{39mm} $\textbf{(C5)}$  $\lim\limits_{n\rightarrow\infty}\frac{\theta_{n}}{\alpha_{n+1}\lambda_{n+1}}=0$; 
		\item[\textbf{(C2)}] $\lim\limits_{n\rightarrow\infty}\frac{1}{\alpha_{n+1}}\big|\frac{1}{\lambda_{n+1}}-\frac{1}{\lambda_{n}}\big|=0$;   \hspace{16mm} $\textbf{(C6)}$ $\frac{\lambda_{n}}{\lambda_{n+1}}\leq\sigma$ for some $\sigma\geq 1$; 
		\item[\textbf{(C3)}]$\lim\limits_{n\rightarrow\infty}\frac{1}{\lambda_{n+1}}\big|1-\frac{\alpha_{n}}{\alpha_{n+1}}\big|=0$;\hspace{20mm} $\textbf{(C7)}$   $\lim\limits_{n\rightarrow\infty}\frac{\beta_{n}}{\alpha_{n+1}}=0$.		
		\item[\textbf{(C4)}]$\lim\limits_{n\rightarrow\infty}\frac{\alpha_{n}}{\lambda_{n}}=0$;  
	\end{description}
\end{condition}	
Under the Assumption \ref{assumption1}, the proximal of $f_{i}$, gradient of $h_{i}$ and the Halpern of $T_{i}$ will be implemented in developing the proposed methods.

\subsection{\textbf{Distributed Accelerated Incremental Type Algorithm}}\label{subsec1}
In this subsection, we introduce an incremental type Halpern and a proximal gradient algorithm with inertial extrapolated and we analyze the convergence of the algorithm for solving MAFPSCOP.

\par\noindent\rule{\textwidth}{1pt}
\begin{alg}\label{algorithm1} (Distributed Accelerated Incremental Type Algorithm)\\  Let $\{\theta_{n}\}$, $\{\lambda_{n}\}$, $\{\beta_{n}\}$ and $\{\alpha_{n}\}$ be real sequences satisfying \textbf{Condition \ref{condition1}}. Choose $x_{1}\in H$, $w^{(i)}_{0},z^{(i)}_{0},u^{(i)}\in H$ $(i\in I)$ and  $d^{(i)}_{1}=-\nabla h_{i}(z^{(i)}_{0})$ $(i\in I)$, and follow the following iterative steps for $n=1,2,3,\ldots$.	
	\begin{description}
		\item [\textbf{Step 1.}] Take $w^{(1)}_{n}=x_{n}$ and compute following finite loop ($i=1,\ldots,M$).
		\[
		\begin{cases}
		\mbox{\textbf{Step 1.1.}   Evaluate } z^{(i)}_{n}=w^{(i)}_{n}+\theta_{n}(w^{(i)}_{n}-w^{(i)}_{n-1}).\\
		\mbox{\textbf{Step 1.2.}   Evaluate } \\ 
		\hspace{24mm}y^{(i)}_{n}=\mbox{prox}_{\lambda_{n}f_{i}}(z^{(i)}_{n}+\lambda_{n}d^{(i)}_{n+1}), \\ 
		\hspace{15mm}\mbox{ where } d^{(i)}_{n+1}=-\nabla h_{i}(z^{(i)}_{n})+\beta_{n}d^{(i)}_{n}.\\
		\mbox{\textbf{Step 1.3.}   Evaluate } w^{(i+1)}_{n}=\alpha_{n}u^{(i)}+(1-\alpha_{n})T_{i}(y^{(i)}_{n}).\\
		\mbox{\textbf{Step 1.4}   If } i<M, \mbox{ set } i=i+1 \mbox{ and go to \textbf{Step 1.1}.}  \\ \hspace{16mm}\mbox{Otherwise, i.e., if $i=M$, go to \textbf{Step 2}}. 
		\end{cases}
		\]	
		\item [\textbf{Step 2.}] Set $$x_{n+1}=w^{(M+1)}_{n}.$$
		\item [\textbf{Step 3.}] Set $n:=n+1$ and go to \textbf{Step 1}.
	\end{description}	
\end{alg}
\par\noindent\rule{\textwidth}{1pt}\\\\
Take the following mild boundedness assumption on the involved variable metrics $\{y_{n}^{(i)}\}$ $(i\in I)$ generated by Algorithm \ref{algorithm1}.  

\begin{assumption}\label{assumption2} The sequence $\{y_{n}^{(i)}\}$ $(i\in I)$ generated by Algorithm \ref{algorithm1} is bounded. 
\end{assumption}	
\begin{lemma} \label{L0}
	Suppose for each user $i\in I$, $X_{i}$ is a bounded, closed and convex subset of $H$ (e.g., $X_{i}$ is a closed ball $\mbox{Cl}(B(w_{(i)},r_{i}))$ in $H$ with a sufficiently large radius $r_{i}$) such that $S\subset X_{i}$. Then if $y^{(i)}_{n}$ of Algorithm \ref{algorithm1} is replaced by the computation of the projection
	\begin{eqnarray}  \label{bounded11}
	y^{(i)}_{n}=P_{X_{i}}(\mbox{prox}_{\lambda_{n}f_{i}}(z^{(i)}_{n}+\lambda_{n}d^{(i)}_{n+1})),	
	\end{eqnarray}  
	or $w^{(i+1)}_{n}$ of Algorithm \ref{algorithm1} is replaced by the computation of the projection		
	\begin{eqnarray}  \label{bounded22}
	w^{(i+1)}_{n}=P_{X_{i}}(\alpha_{n}u^{(i)}+(1-\alpha_{n})T_{i}(y^{(i)}_{n})),
	\end{eqnarray}  	  
	the sequence $\{y_{n}^{(i)}\}$ $(i\in I)$ generated by the resulting algorithm is bounded.	
\end{lemma}

Based on Lemma \ref{L0}, we can see that if $S=\bigcap_{i=1}^{M}FixT_{i}$ is bounded (e.g., $FixT_{i_{0}}$ is bounded for some $i_{0}\in I$), the existence of a bounded, closed and convex subset $X_{i}$ of $H$ is  granted, and hence Assumption \ref{assumption2} can be eliminated by replacing $y^{(i)}_{n}$ by (\ref{bounded11}) or replacing $w^{(i+1)}_{n}$ by (\ref{bounded22}) in Algorithm \ref{algorithm1}. Note that also the convergence analysis of the new algorithm obtained by replacing $y^{(i)}_{n}$ by (\ref{bounded11}) or replacing $w^{(i+1)}_{n}$ by (\ref{bounded22}) in Algorithm \ref{algorithm1} is the same as the convergence analysis of Algorithm \ref{algorithm1} under Assumptions \ref{assumption1} and \ref{assumption2}. We next examine the convergence analysis of Algorithm \ref{algorithm1} under Assumptions \ref{assumption1} and  \ref{assumption2}.

\begin{lemma}\label{L1}
	The sequences $\{T_{i}(y^{(i)}_{n})\}$ $(i\in I)$,  $\{x_{n}\}$, $\{w_{n}^{(i)}\}$ $(i\in I)$, $\{z_{n}^{(i)}\}$ $(i\in I)$ and $\{d_{n}^{(i)}\}$ $(i\in I)$ generated by Algorithm \ref{algorithm1} are bounded.
\end{lemma}

\begin{lemma}\label{L2}
	Let $\{x_{n}\}$, $\{y^{(i)}_{n}\}$ $(i\in I)$, $\{w_{n}^{(i+1)}\}$ $(i\in \{0\}\cup I)$, $\{z_{n}^{(i)}\}$ $(i\in I)$ and $\{d_{n}^{(i)}\}$ $(i\in I)$ be sequences generated by Algorithm \ref{algorithm1}. Then the following holds: 
	\item{(a).}
	$\lim\limits_{n\rightarrow\infty}\frac{\|x_{n+1}-x_{n}\|}{\lambda_{n}}=\lim\limits_{n\rightarrow\infty}\|x_{n+1}-x_{n}\|=0.
	$
	\item{(b).} For $\bar{x}\in S$ and $i\in I$, there exist $Q_{1}, Q_{2}\geq 0$ such that
	\begin{eqnarray}  
	\|w^{(i+1)}_{n}-\bar{x}\|^{2}\nonumber&\leq&\|w^{(i)}_{n}-\bar{x}\|^{2}+\theta_{n}Q_{1}+\alpha_{n}Q_{2}-(1-\alpha_{n})\|T_{i}(y^{(i)}_{n})-y^{(i)}_{n}\|^{2}\nonumber\\&&-\|z^{(i)}_{n}-y^{(i)}_{n}\|^{2}+2\lambda_{n}(f_{i}(\bar{x})-f_{i}(y^{(i)}_{n}))+2\lambda_{n}\langle  d^{(i)}_{n+1},y^{(i)}_{n}-\bar{x}\rangle.\nonumber 
	\end{eqnarray}
\end{lemma}

\begin{lemma}\label{L3}
	For the sequences $\{x_{n}\}$, $\{y^{(i)}_{n}\}$ $(i\in I)$, $\{w_{n}^{(i)}\}$ $(i\in I)$ and $\{z_{n}^{(i)}\}$ $(i\in I)$ generated by Algorithm \ref{algorithm1}, we have
	\begin{description}
		\item[(a).]$ \lim\limits_{n\rightarrow\infty}\|z^{(i)}_{n}-y^{(i)}_{n}\|= \lim\limits_{n\rightarrow\infty}\|x_{n}-w^{(i)}_{n}\|=0$ for all $i\in I$;
		\item[(b).]$\lim\limits_{n\rightarrow\infty}\|x_{n}-z^{(i)}_{n}\|=\lim\limits_{n\rightarrow\infty}\|T_{i}(x_{n})-x_{n}\|=0$ for all $i\in I$.
	\end{description}	
\end{lemma}

\begin{theorem}\label{T1} For $\psi$ be given as in MAFPSCOP, i.e., $\psi=\sum_{i=1}^{M}\mathcal{G}_{i}$ where $\mathcal{G}_{i}=f_{i}+h_{i}$, and for the sequence  $\{w_{n}^{(i)}\}$ ($i\in I$) generated by Algorithm \ref{algorithm1} we have the following.
	\begin{description}	
		\item{(a).}	$\limsup_{n\rightarrow\infty}\psi(w_{n}^{(1)})\leq \psi(\bar{x})$ for all $\bar{x}\in S$.
		\item{(b).} Any weak
		sequential cluster point of $\{w_{n}^{(i)}\}$ ($i\in I$) belongs to $\Omega$.
		\item{(c).} If $\psi$ is strictly convex, $\{w_{n}^{(i)}\}$ ($i\in I$) weakly converges to a point in $\Omega$.
	\end{description}	
\end{theorem}

\begin{proof} $(a).$ Let $\bar{x}\in S=\bigcap_{i=1}^{M}FixT_{i}$. Then from Lemma \ref{L2} $(b)$, we have 
	\begin{eqnarray}\label{eq:25}
	&&\|w^{(i+1)}_{n}-\bar{x}\|^{2}\leq\|w^{(i)}_{n}-\bar{x}\|^{2}+\theta_{n}Q_{1}+\alpha_{n}Q_{2}+2\lambda_{n}(f_{i}(\bar{x})-f_{i}(y^{(i)}_{n}))\nonumber\\&&\hspace{21mm}+2\lambda_{n}\langle  d^{(i)}_{n+1},y^{(i)}_{n}-\bar{x}\rangle
	\end{eqnarray}	
	for each $i\in I$. By summing both side of the inequality (\ref{eq:25}) over $i$ from 1 to $M$ and noting that $w^{(M+1)}_{n}=x_{n+1}$ and $w^{(1)}_{n}=x_{n}$ gives
	\begin{eqnarray}\label{eq:26}
	\|x_{n+1}-\bar{x}\|^{2}\nonumber&\leq&\|x_{n}-\bar{x}\|^{2}+\theta_{n}MQ_{1}+\alpha_{n}MQ_{2}+2\lambda_{n}\sum\limits_{i=1}^{M}(f_{i}(\bar{x})-f_{i}(z_{n}^{(i)}))\nonumber\\&&\hspace{25mm}+2\lambda_{n}\sum\limits_{i=1}^{M}\langle d^{(i)}_{n+1},y_{n}^{(i)}-\bar{x}\rangle.
	\end{eqnarray}
	Using the definition $d^{(i)}_{n+1}=-\nabla h_{i}(z^{(i)}_{n})+\beta_{n}d^{(i)}_{n}$in the algorithm, it holds that
	\begin{eqnarray}\label{eq:27}
	\nonumber&&2\langle d^{(i)}_{n+1},y_{n}^{(i)}-\bar{x}\rangle=2\langle d^{(i)}_{n+1},z_{n}^{(i)}-\bar{x}\rangle+2\langle d^{(i)}_{n+1},y_{n}^{(i)}-z_{n}^{(i)}\rangle\nonumber\\&&\hspace{15mm}=2\langle\nabla h_{i}(z^{(i)}_{n}),\bar{x}-z_{n}^{(i)}\rangle+2\beta_{n}\langle d^{(i)}_{n},z_{n}^{(i)}-\bar{x}\rangle+2\langle d^{(i)}_{n+1},y_{n}^{(i)}-z_{n}^{(i)}\rangle\nonumber\\&&\hspace{15mm}\leq2(h_{i}(\bar{x})-h_{i}(z^{(i)}_{n}))+\beta_{n}Q_{6}+\|y_{n}^{(i)}-z_{n}^{(i)}\|Q_{7}
	\end{eqnarray}
	for $i\in I$, where $Q_{6}=\max_{i\in I}(\sup\{2\langle d^{(i)}_{n},z_{n}^{(i)}-\bar{x}\rangle:n\in\mathbb{N}\})<\infty$ and $Q_{7}=\max_{i\in I}(\sup\{2 \|d^{(i)}_{n+1}\|:n\in\mathbb{N}\})<\infty$.
	Hence, from (\ref{eq:26}), (\ref{eq:27}), $\mathcal{G}_{i}=f_{i}+h_{i}$ and $\psi=\sum_{i=1}^{M}\mathcal{G}_{i}$ it holds
	\begin{eqnarray}\label{eq:28}
	\nonumber&&\|x_{n+1}-\bar{x}\|^{2}\leq\|x_{n}-\bar{x}\|^{2}+\theta_{n}MQ_{1}+\alpha_{n}MQ_{2}+2\lambda_{n}\sum\limits_{i=1}^{M}(f_{i}(\bar{x})-f_{i}(y_{n}^{(i)}))\nonumber\\&&\hspace{15mm}+2\lambda_{n}\sum\limits_{i=1}^{M}(h_{i}(\bar{x})-h_{i}(z^{(i)}_{n}))+\beta_{n}\lambda_{n}MQ_{6}+\lambda_{n}Q_{7}\sum\limits_{i=1}^{M}\|y_{n}^{(i)}-z_{n}^{(i)}\|\nonumber\\&&\hspace{6mm}=\|x_{n}-\bar{x}\|^{2}+\theta_{n}MQ_{1}+\alpha_{n}MQ_{2}+\beta_{n}\lambda_{n}MQ_{6}+\lambda_{n}Q_{7}\sum\limits_{i=1}^{M}\|y_{n}^{(i)}-z_{n}^{(i)}\|\nonumber\\&&\hspace{20mm}+2\lambda_{n}\sum\limits_{i=1}^{M}(\mathcal{G}_{i}(\bar{x})-\mathcal{G}_{i}(x_{n}))+2\lambda_{n}\sum\limits_{i=1}^{M}(f_{i}(x_{n})-f_{i}(y^{(i)}_{n}))\nonumber\\&&\hspace{20mm}+2\lambda_{n}\sum\limits_{i=1}^{M}(h_{i}(x_{n})-h_{i}(y^{(i)}_{n}))+2\lambda_{n}\sum\limits_{i=1}^{M}(h_{i}(y^{(i)}_{n})-h_{i}(z^{(i)}_{n}))\nonumber\\&&\hspace{6mm}=\|x_{n}-\bar{x}\|^{2}+\theta_{n}MQ_{1}+\alpha_{n}MQ_{2}+\beta_{n}\lambda_{n}MQ_{6}+\lambda_{n}Q_{7}\sum\limits_{i=1}^{M}\|y_{n}^{(i)}-z_{n}^{(i)}\|\nonumber\\&&\hspace{20mm}+2\lambda_{n}(\psi(\bar{x})-\psi(x_{n}))+2\lambda_{n}\sum\limits_{i=1}^{M}(f_{i}(x_{n})-f_{i}(y^{(i)}_{n}))\\&&\hspace{20mm}+2\lambda_{n}\sum\limits_{i=1}^{M}(h_{i}(x_{n})-h_{i}(y^{(i)}_{n}))+2\lambda_{n}\sum\limits_{i=1}^{M}(h_{i}(y^{(i)}_{n})-h_{i}(z^{(i)}_{n})).\nonumber
	\end{eqnarray}	
	Moreover, Assumption \ref{assumption1} (A1), the definition of $\partial f_{i}$ and Proposition \ref{L1} $(d)$ guarantees that there is $r_{n}^{(i)}\in\partial f_{i}(y_{n}^{(i)})$ such that 
	$f_{i}(x_{n})-f_{i}(y_{n}^{(i)})\leq\langle x_{n}-y_{n}^{(i)},r_{n}^{(i)}\rangle,$
	and which then implies by Lemma \ref{L2}, we have
	$2(f_{i}(x_{n})-f_{i}(y_{n}^{(i)}))\leq Q_{8}\|x_{n}-y_{n}^{(i)}\|,$ where $Q_{8}=\max_{i\in I}(\sup\{2\|r_{n}^{(i)}\|:n\in\mathbb{N}\})<\infty$. In addition, since $\partial h_{i}(x_{n})=\{\nabla h_{i}(x_{n})\}$ and  $\partial h_{i}(y^{(i)}_{n})=\{\nabla h_{i}(y^{(i)}_{n})\}$ (obtained as a result of Assumption \ref{assumption1} $(A2)$) and the sequences $\{\nabla h_{i}(x_{n})\}$ and $\{\nabla h_{i}(y^{(i)}_{n})\}$ are bounded (obtained as a result of $\{x_{n}\}$ and $\{y^{(i)}_{n}\}$ are bounded from Lemma \ref{L2} and $\nabla h_{i}$ is Lipschitz continuous from Assumption \ref{assumption1} $(A2)$), we obtain $2(h_{i}(x_{n})-h_{i}(y_{n}^{(i)}))\leq2\langle x_{n}-y_{n}^{(i)},\nabla h_{i}(x_{n})\rangle\leq Q_{9}\|x_{n}-y_{n}^{(i)}\|$  and   $2(h_{i}(y_{n}^{(i)})-h_{i}(z_{n}^{(i)}))\leq2\langle y_{n}^{(i)}-z_{n}^{(i)},\nabla h_{i}(y_{n}^{(i)})\rangle\leq Q_{10}\|y_{n}^{(i)}-z_{n}^{(i)}\|,$ where $Q_{9}=\max_{i\in I}(\sup\{2\|\nabla h_{i}(x_{n})\|:n\in\mathbb{N}\})<\infty$ and $Q_{10}=\max_{i\in I}(\sup\{2\|\nabla h_{i}(y_{n}^{(i)})\|:n\in\mathbb{N}\})<\infty$. Therefore, (\ref{eq:28}) gives
	\begin{eqnarray}
	\|x_{n+1}-\bar{x}\|^{2}\nonumber&\leq&\|x_{n}-\bar{x}\|^{2}+\theta_{n}MQ_{1}+\alpha_{n}MQ_{2}+\beta_{n}\lambda_{n}MQ_{6}+2\lambda_{n}(\psi(\bar{x})-\psi(x_{n}))\nonumber\\&&+\lambda_{n}Q_{11}\sum\limits_{i=1}^{M}\|y_{n}^{(i)}-z_{n}^{(i)}\|+\lambda_{n}Q_{12}\sum\limits_{i=1}^{M}\|x_{n}-y_{n}^{(i)}\|\nonumber,
	\end{eqnarray}	
	where $Q_{11}=Q_{7}+Q_{10}$ and $Q_{12}=Q_{8}+Q_{9}$, and so, by rearranging the terms, we arrive at
	\begin{eqnarray}\label{eq:29}
	2(\psi(x_{n})-\psi(\bar{x}))\nonumber&\leq&\frac{1}{\lambda_{n}}(\|x_{n}-\bar{x}\|^{2}-\|x_{n+1}-\bar{x}\|^{2})+\frac{\theta_{n}}{\lambda_{n}}MQ_{1}+\frac{\alpha_{n}}{\lambda_{n}}MQ_{2}\nonumber\\&&+\beta_{n}MQ_{6}+Q_{11}\sum\limits_{i=1}^{M}\|y_{n}^{(i)}-z_{n}^{(i)}\|+Q_{12}\sum\limits_{i=1}^{M}\|x_{n}-z_{n}^{(i)}\|\nonumber\\&\leq&\frac{\|x_{n}-x_{n+1}\|}{\lambda_{n}}Q_{13}+\frac{\theta_{n}}{\lambda_{n}}MQ_{1}+\frac{\alpha_{n}}{\lambda_{n}}MQ_{2}+\beta_{n}MQ_{6}\nonumber\\&&+Q_{11}\sum\limits_{i=1}^{M}\|y^{(i)}_{n}-z_{n}^{(i)}\|+Q_{12}\sum\limits_{i=1}^{M}\|x_{n}-z_{n}^{(i)}\|,
	\end{eqnarray}	
	where  $Q_{13}=\sup\{\|x_{n}-\bar{x}\|+\|x_{n+1}-\bar{x}\|:n\in\mathbb{N}\}<\infty$.  Hence, (\ref{eq:29}), Lemma \ref{L2} $(a)$, Lemma \ref{L3} $(b)$ and Condition \ref{condition1} ensure that
	$
	\limsup_{n\rightarrow\infty}(\psi(x_{n})-\psi(\bar{x}))\leq 0
	$ for all $\bar{x}\in S$.
	Therefore, 
	\begin{eqnarray}\label{eq:29aa}
	\limsup_{n\rightarrow\infty}\psi(w_{n}^{(1)})=\limsup_{n\rightarrow\infty}\psi(x_{n})\leq \psi(\bar{x}) \mbox{ for all } \bar{x}\in S.
	\end{eqnarray}
	\item{$(b)$.} The boundedness of $\{x_{n}\}$ in the Hilbert space $H$ guarantees the existence of a subsequence $\{x_{n_{l}}\}$ such that $\{x_{n_{l}}\}$ weakly converges to $p\in H$. Assume $p\notin S$. This implies that there exists $i_{0}\in I$ with $p\notin FixT_{i_{0}}$, i.e., $T_{i_{0}}(p)\neq p$. Hence, Opial’s condition, $\lim_{n\rightarrow\infty}\|x_{n}-T_{i_{0}}(x_{n})\|=0$ from Lemma \ref{L3}, and the nonexpansivity of $T_{i_{0}}$ imply that 
	\begin{eqnarray}
	\liminf_{l\rightarrow\infty}\|x_{n_{l}}-p\|\nonumber&<&\liminf_{l\rightarrow\infty}\|x_{n_{l}}-T_{i_{0}}(p)\|\nonumber\\&=&\liminf_{l\rightarrow\infty}\|x_{n_{l}}-T_{i_{0}}(x_{n_{l}})+T_{i_{0}}(x_{n_{l}})-T_{i_{0}}(p)\|\nonumber\\&\leq&\liminf_{l\rightarrow\infty}(\|x_{n_{l}}-T_{i_{0}}(x_{n_{l}})\|+\|T_{i_{0}}(x_{n_{l}})-T_{i_{0}}(p)\|)\nonumber\\&=&\liminf_{l\rightarrow\infty}\|T_{i_{0}}(x_{n_{l}})-T_{i_{0}}(p)\|\leq\liminf_{l\rightarrow\infty}\|x_{n_{l}}-p\|.\nonumber
	\end{eqnarray}
	This is a contradiction, and hence it must be the case that $p\in S$. The (\ref{eq:29aa}),  $x_{n_{l}}\rightharpoonup p\in S$, and the
	convexity and continuity of $\psi$ guarantee that 
	$
	\psi(p)\leq\liminf_{l\rightarrow\infty}\psi(x_{n_{l}})\leq\limsup_{l\rightarrow\infty}\psi(x_{n_{l}})\leq \psi(\bar{x}) \mbox{ for all } \bar{x}\in S.
	$
	Therefore, $p$ the solution of MAFPSCOP, i.e., $p\in\Omega=\big\{\hat{x}\in S: \psi (\hat{x})=\min_{x\in S}\psi (x) \big\}$. This implies also that any weak sequential cluster point
	of $\{x_{n}\}$ is in $\Omega$, and hence utilizing the result $\|x_{n}-w^{(i)}_{n}\|\rightarrow 0$ ($i\in I$) as $n\rightarrow\infty$ from Lemma \ref{L3} $(a)$, we conclude that any weak sequential cluster point of $\{w_{n}^{(i)}\}$ ($i\in I$) is in $\Omega$.
	\item{$(c)$.} Since $\psi:H\rightarrow\mathbb{R}$ is strictly convex, the uniqueness of the solution, denoted by $\bar{x}$, of MAFPSCOP is guaranteed, i.e., $\Omega=\{\bar{x}\}$. Since  $\{x_{n}\}$ is bounded there exists a subsequence $\{x_{n_{l}}\}$ of $\{x_{n}\}$ such that $\{x_{n_{l}}\}$ weakly converges to $p\in H$, and by $(b)$, we know that $p$ is in $\Omega$. Hence, the uniqueness of the solution of MAFPSCOP gives $p=\bar{x}$. This implies that any weakly converging subsequence of $\{x_{n}\}$ weakly converges to $\bar{x}$ with
	$\langle \bar{x},w\rangle =\lim_{l\rightarrow\infty} \langle x_{n_{l}}, w\rangle$, $w\in H$. Therefore, any subsequence of $\{x_{n}\}$ weakly converges to $\bar{x}$. Hence, $\{x_{n}\}$ weakly converges to $\bar{x}$, and this together with the result $\|x_{n}-w^{(i)}_{n}\|\rightarrow 0$ ($i\in I$) as $n\rightarrow\infty$ in Lemma \ref{L3} $(a)$ implies that the sequence $\{w_{n}^{(i)}\}$ ($i\in I$) weakly converges to unique solution point $\bar{x}$ of MAFPSCOP.
\end{proof}
\subsection{\textbf{Distributed Accelerated Parallel Type Algorithm}}\label{subsec2}
In this subsection, we present accelerated Halpern and 
proximal gradient algorithm constructed in a parallel computing manner.
\par\noindent\rule{\textwidth}{1pt}
\begin{alg}\label{algorithm2} (Distributed Accelerated Parallel Type Algorithm)\\
	Let $\{\theta_{n}\}$, $\{\lambda_{n}\}$, $\{\beta_{n}\}$ and $\{\alpha_{n}\}$ be real sequences satisfying Condition \ref{condition1}.  Choose $z_{0},x_{0},x_{1}\in H_{1}$, $u^{(i)}\in H$ $(i\in I)$  and $d^{(i)}_{1}=-\nabla h_{i}(z_{0})$ $(i\in I)$ and follow the following iterative steps for $n=1,2,3,\ldots$.
	\begin{description}
		\item [\textbf{Step 1.}] Find $z_{n}=x_{n}+\theta_{n}(x_{n}-x_{n-1})$
		and compute the following finite loop ($i=1,\ldots,M$).
		
		\[
		\begin{cases}
		\mbox{\textbf{Step 1.1.}   Evaluate }\\
		
		\hspace{25mm} y^{(i)}_{n}=\mbox{prox}_{\lambda_{n}f_{i}}(z_{n}+\lambda_{n}d^{(i)}_{n+1}),\\
		
		\hspace{15mm}\mbox{ where } d^{(i)}_{n+1}=-\nabla h_{i}(z_{n})+\beta_{n}d^{(i)}_{n}.\\
		\mbox{\textbf{Step 1.2.}   Evaluate } w^{(i+1)}_{n}=\alpha_{n}u^{(i)}+(1-\alpha_{n})T_{i}(y^{(i)}_{n}).\\
		\mbox{\textbf{Step 1.3.}   If } i<M, \mbox{ set } i=i+1 \mbox{ and go to \textbf{Step 1.1}.}  \\ \hspace{16mm}\mbox{Otherwise, i.e., if $i=M$, go to \textbf{Step 2}}. 
		\end{cases}
		\]	
		\item [\textbf{Step 2.}] Evaluate $$x_{n+1}=\frac{1}{M}\sum\limits_{i=1}^{M} w^{(i+1)}_{n}.$$
		\item [\textbf{Step 3.}] Set $n:=n+1$ and go to \textbf{Step 1}.
	\end{description}	
\end{alg}
\par\noindent\rule{\textwidth}{1pt}\\\\
We make the following assumption.

\begin{assumption}\label{assumption3} The sequence $\{y_{n}^{(i)}\}$ $(i\in I)$ generated by Algorithm \ref{algorithm2} is a bounded.
\end{assumption}
We now present the convergence analysis of Algorithm \ref{algorithm2} under the Assumptions \ref{assumption1} and  \ref{assumption3}.\\

\begin{lemma}\label{L4} The sequences $\{T_{i}(y^{(i)}_{n})\}$ $(i\in I)$,  $\{x_{n}\}$, $\{w_{n}^{(i)}\}$ $(i\in I)$, $\{z_{n}\}$ and $\{d_{n}^{(i)}\}$ $(i\in I)$ generated by Algorithm \ref{algorithm2} are bounded.
\end{lemma}
\begin{lemma}\label{L5}
	Let $\{x_{n}\}$, $\{y^{(i)}_{n}\}$ $(i\in I)$, $\{w_{n}^{(i+1)}\}$ $(i\in I)$, $\{z_{n}\}$ and $\{d_{n}^{(i)}\}$ $(i\in I)$ be sequences generated by Algorithm \ref{algorithm2}. Then the following holds: 
	\item{$(a)$.} $\lim\limits_{n\rightarrow\infty}\frac{\|x_{n+1}-x_{n}\|}{\lambda_{n}}=\lim\limits_{n\rightarrow\infty}\|x_{n+1}-x_{n}\|=0.
	$
	\item{$(b)$.} For $\bar{x}\in S$ and $i\in I$, there exist $D_{1}, D_{2}\geq 0$ such that
	\begin{eqnarray}\|w^{(i+1)}_{n}-\bar{x}\|^{2}\nonumber&\leq&\|x_{n}-\bar{x}\|^{2}+\theta_{n}D_{1}+2\lambda_{n}(f_{i}(\bar{x})-f_{i}(y^{(i)}_{n}))+2\lambda_{n}\langle  d^{(i)}_{n+1},y^{(i)}_{n}-\bar{x}\rangle\nonumber\\&&+\alpha_{n}D_{2}-(1-\alpha_{n})\|T_{i}(y^{(i)}_{n})-y^{(i)}_{n}\|^{2}-\|z_{n}-y^{(i)}_{n}\|^{2}.\nonumber
	\end{eqnarray}	
	
\end{lemma}
\begin{lemma}\label{L6}
	For sequences $\{T_{i}(x^{(i)}_{n})\}$ $(i\in I)$,  $\{x_{n}\}$, $\{z_{n}\}$ and $\{y_{n}^{(i)}\}$ $(i\in I)$ generated by Algorithm \ref{algorithm2}, we have for each $i\in I$ that
	$$ \lim\limits_{n\rightarrow\infty}\|z_{n}-y^{(i)}_{n}\|= \lim\limits_{n\rightarrow\infty}\|x_{n}-y^{(i)}_{n}\|=\lim\limits_{n\rightarrow\infty}\|T_{i}(x_{n})-x_{n}\|=0.$$
\end{lemma}

\begin{theorem}\label{T2}  For $\psi$ given as in MAFPSCOP, i.e., $\psi=\sum_{i=1}^{M}\mathcal{G}_{i}$ where $\mathcal{G}_{i}=f_{i}+h_{i}$, and for the sequence $\{x_{n}\}$ generated by Algorithm \ref{algorithm2} we have the following.
	\item{$(a)$.}	$\limsup_{n\rightarrow\infty}\psi(x_{n})\leq \psi(\bar{x})$ for all $\bar{x}\in S$.	
	\item{$(b)$.} Any weak
	sequential cluster point of $\{x_{n}\}$  belongs to $\Omega$.
	\item{$(c)$.} If $\psi$ is strictly convex, $\{x_{n}\}$ weakly converges to a point in $\Omega$.
\end{theorem}
\begin{proof}$(a)$. By $x_{n+1}=\frac{1}{M}\sum_{i=1}^{M}w^{(i+1)}_{n}$, $\bar{x}=\frac{1}{M}\sum_{i=1}^{M}\bar{x}$, and convexity of $\|.\|^{2}$, we get $	\|x_{n+1}-\bar{x}\|^{2}\leq\frac{1}{M}\sum_{i=1}^{M}\|w^{(i+1)}_{n}-\bar{x}\|^{2}$, and thus from Lemma \ref{L5} we get 
	\begin{eqnarray}\label{eq:bb26}
	\|x_{n+1}-\bar{x}\|^{2}\nonumber&\leq&\|x_{n}-\bar{x}\|^{2}+\theta_{n}D_{1}+\alpha_{n}D_{2}+\frac{2\lambda_{n}}{M}\sum\limits_{i=1}^{M}(f_{i}(\bar{x})-f_{i}(z_{n}))\nonumber\\&&+\frac{2\lambda_{n}}{M}\sum\limits_{i=1}^{M}\langle d^{(i)}_{n+1},y_{n}^{(i)}-\bar{x}\rangle.
	\end{eqnarray}
	Moreover, from $d^{(i)}_{n+1}=-\nabla h_{i}(z_{n})+\beta_{n}d^{(i)}_{n}$, we get $\langle d^{(i)}_{n+1},y_{n}^{(i)}-\bar{x}\rangle=\langle \nabla h_{i}(z_{n}),\bar{x}-z_{n}\rangle+\beta_{n}\langle d^{(i)}_{n},z_{n}-\bar{x}\rangle+\langle d^{(i)}_{n+1},y_{n}-z_{n}^{(i)}\rangle$, and thus
	\begin{eqnarray}\label{eq:bb27}
	\frac{2}{M}\langle d^{(i)}_{n+1},y_{n}^{(i)}-\bar{x}\rangle\leq\frac{2}{M}(h_{i}(\bar{x})-h_{i}(z_{n}))+\beta_{n}\frac{1}{M}D_{6}+\|y_{n}^{(i)}-z_{n}\|D_{7},
	\end{eqnarray}
	where $D_{6}=\max_{i\in I}(\sup\{2\langle d^{(i)}_{n},z_{n}-\bar{x}\rangle:n\in\mathbb{N}\})<\infty$ and $D_{7}=\max_{i\in I}(\sup\{\frac{2}{M} \|d^{(i)}_{n+1}\|:n\in\mathbb{N}\})<\infty$.
	Therefore, from (\ref{eq:bb26}) and (\ref{eq:bb27}), we arrive at
	\begin{eqnarray}\label{eq:bb28}
	\|x_{n+1}-\bar{x}\|^{2}\nonumber&\leq&\|x_{n}-\bar{x}\|^{2}+\theta_{n}D_{1}+\alpha_{n}D_{2}+\beta_{n}\lambda_{n}D_{6}+\lambda_{n}D_{7}\sum\limits_{i=1}^{M}\|z_{n}-y_{n}^{(i)}\|\nonumber\\&&+\frac{2\lambda_{n}}{M}(\psi(\bar{x})-\psi(x_{n}))+\frac{2\lambda_{n}}{M}\sum\limits_{i=1}^{M}(f_{i}(x_{n})-f_{i}(y^{(i)}_{n}))\\&&+\frac{2\lambda_{n}}{M}\sum\limits_{i=1}^{M}(h_{i}(x_{n})-h_{i}(y^{(i)}_{n}))+\frac{2\lambda_{n}}{M}\sum\limits_{i=1}^{M}(h_{i}(y^{(i)}_{n})-h_{i}(z_{n})).\nonumber
	\end{eqnarray}
	Note that for all $n\in\mathbb{N}$, we have	$\frac{2}{M}(f_{i}(x_{n})-f_{i}(y_{n}^{(i)}))\leq V_{1}\|x_{n}-y_{n}^{(i)}\|$, $\frac{2}{M}(h_{i}(x_{n})-h_{i}(y_{n}^{(i)}))\leq V_{2}\|x_{n}-y_{n}^{(i)}\|$ and $\frac{2}{M}(h_{i}(y_{n}^{(i)})-h_{i}(z_{n}^{(i)}))\leq V_{3}\|y_{n}^{(i)}-z_{n}\|$, where $V_{1}=\max_{i\in I}(\sup\{\frac{2}{M}\|q^{(i)}_{n}\|:n\in\mathbb{N}\})<\infty$ for $q^{(i)}_{n}\in\partial f_{i}(y_{n}^{(i)})$, $V_{2}=\max_{i\in I}(\sup\{\frac{2}{M}\|\nabla h_{i}(x_{n})\|:n\in\mathbb{N}\})$ and $V_{3}=\max_{i\in I}(\sup\{\frac{2}{M}\|\nabla h_{i}(y_{n}^{(i)})\|:n\in\mathbb{N}\})$. Hence, (\ref{eq:bb28}) becomes
	\begin{eqnarray}
	\|x_{n+1}-\bar{x}\|^{2}\nonumber&\leq&\|x_{n}-\bar{x}\|^{2}+\theta_{n}D_{1}+\alpha_{n}D_{2}+\beta_{n}\lambda_{n}D_{6}+\frac{2\lambda_{n}}{M}(\psi(\bar{x})-\psi(x_{n}))\nonumber\\&&+\lambda_{n}D_{8}\sum\limits_{i=1}^{M}\|x_{n}-y_{n}^{(i)}\|+\lambda_{n}D_{9}\sum\limits_{i=1}^{M}\|y^{(i)}_{n}-z_{n}\|,\nonumber
	\end{eqnarray}	
	where $D_{8}=V_{1}+V_{2}$ and $D_{9}=D_{3}+V_{7}$, and this implies that
	\begin{eqnarray}
	\frac{2}{M}(\psi(x_{n})-\psi(\bar{x}))\nonumber&\leq&\frac{\|x_{n}-x_{n+1}\|}{\lambda_{n}}D_{10}+\frac{\theta_{n}}{\lambda_{n}}D_{1}+\frac{\alpha_{n}}{\lambda_{n}}D_{2}+\beta_{n}D_{6}\nonumber\\&&+D_{8}\sum\limits_{i=1}^{M}\|x_{n}-y_{n}^{(i)}\|+D_{9}\sum\limits_{i=1}^{M}\|y^{(i)}_{n}-z_{n}\|,\nonumber
	\end{eqnarray}	
	where $D_{10}=\sup\{\|x_{n}-\bar{x}\|+\|x_{n+1}-\bar{x}\|:n\in\mathbb{N}\}<\infty$. 	
	Hence, 
	$
	\limsup_{n\rightarrow\infty}\psi(x_{n})\leq \psi(\bar{x})$	for all $x\in S$.\\
	The proof of $(b)$ and $(c)$ of the Theorem are omitted as it is similar to the proof of $(b)$ and $(c)$ of Theorem \ref{T1}.
\end{proof}
\section{Numerical result}\label{Sec4}
In this section, we give two examples of MAFPSCOP to show some numerical results and to compare our two algorithms (Algorithm \ref{algorithm1} and \ref{algorithm2}) with proximal methods in \cite{6} and the gradient methods in \cite{5} where Algorithm \ref{algorithm1} is obtained by replacing $y^{(i)}_{n}$ by 
$
y^{(i)}_{n}=P_{X_{i}}(\mbox{prox}_{\lambda_{n}f_{i}}(z^{(i)}_{n}+\lambda_{n}d^{(i)}_{n+1}))$  and Algorithm \ref{algorithm1} is obtained by replacing $y^{(i)}_{n}$ by 
$
y^{(i)}_{n}=P_{X_{i}}(\mbox{prox}_{\lambda_{n}f_{i}}(z_{n}+\lambda_{n}d^{(i)}_{n+1}))$, respectively, for the given bounded, convex and closed subset $X_{i}$ of $H$ with $S=\bigcap_{i=1}^{M}FixT_{i}\subset X_{i}$. We perform the numerical experiment for different real parameters $\alpha_{n}=\frac{10^{-1}}{(n+\hat{a})^{a}}$, $\theta_{n}=\frac{t(10^{-1})}{(n+q)^{b}}$, $\lambda_{n}=\frac{10^{-1}}{(n+\hat{c})^{c}}$ and $\beta_{n}=\frac{10^{-1}}{(n+\hat{d})^{d}}$ where $a$, $\hat{a}$, $b$, $q$, $t$, $c$, $\hat{c}$, $d$ and $\hat{d}$, are real numbers taken so that $\{\alpha_{n}\}$, $\{\theta_{n}\}$, $\{\lambda_{n}\}$ and $\{\beta_{n}\}$ satisfy Condition \ref{condition1}, i.e., $c\in (0,1/2)$, $a\in(c,1-c)$, $\hat{a}\geq\hat{c}>0$, $q\geq\max\{\hat{a}+1,\hat{c}+1\}$, $b>a+c$, $\hat{d}\geq\hat{a}+1$, $d>a$. In the tables we report the results of the average CPU time execution in seconds (CPU(s)) and number of iterations ($n$) averaged over the 5 instances for the stopping criteria $\frac{E(n)}{10\max\{\|x_{1}\|,\|x_{0}\|\}}\leq TOL$ where  $E(n)=\|x_{n}-x_{n-1}\|$.

All codes were written in MATLAB and is performed on HP laptop with Intel(R) Core(TM) i5-7200U CPU @ 250GHz 2.70GHz and RAM 4.00GB.
\begin{example}	[\textbf{Comparison of Algorithm \ref{algorithm1} and \ref{algorithm2}}]
	For $\psi:\mathbb{R}^{3}\rightarrow\mathbb{R}$ consider 
	\begin{eqnarray}	\label{numer1}
	\begin{array}{l@{}l}
	&{}\mbox{minimize } \psi(x)=\sum\limits_{i=1}^{2}(\delta_{C_{i}}(x)+\eta_{i}\|x-a_{i}\|_{2}^{2})\\&\hspace{0mm}\mbox{subject to } x\in\bigcap\limits_{i=1}^{2}\arg\min g_{i},
	\end{array}
	\end{eqnarray}
	where $\delta_{C_{i}}$ $(i\in\{1,2\})$ is an indicator function of a set $C_{i}$ a subset of $\mathbb{R}^{3}$ for
	$C_{1}=\big\{x=(x_{1},x_{2},x_{3})^{T}\in \mathbb{R}^{3}:x_{1}+x_{3}\leq 1\big\}$ and $C_{2}=\big\{x=(x_{1},x_{2},x_{3})^{T}\in \mathbb{R}^{3}:4x_{1}+x_{2}+4x_{3}\leq 4\big\}$,
	$a_{1}=(0,0,0)^{T}$, $a_{2}=(2,3,4)^{T}$, $\eta_{1}=1$, $\eta_{2}=2$,  $g_{1}(x)=\frac{1}{2}\|A(x)\|_{2}^{2}$ for $3\times 3$ matrix $A$ and $A^{T}A$ is symmetric positive definite, and $g_{2}((x_{1},x_{2},x_{3})^{T})=\sum_{k=1}^{3}l(x_{k})$ for $l(x_{k})=\max\{|x_{k}|-1,0\}$ $(k\in\{1,2,3\})$.
	
	We implement our methods to solve the problem (\ref{numer1}) in view of the following two settings of (\ref{numer1}) formulated in the form of MAFPSCOP by setting $\mathcal{G}_{i}(x)=f_{i}(x)+h_{i}(x)$ and $T_{i}=\mbox{prox}_{1 g_{i}}=\mbox{prox}_{g_{i}}$  
	where $f_{i}(x)=\delta_{C_{i}}(x)$ and $h_{i}(x)=\eta_{i}\|x-a_{i}\|_{2}^{2}$ $(i\in I=\{1,2\})$. Note that under this setting,
	$\bigcap_{i=1}^{2}FixT_{i}=\bigcap_{i=1}^{2}\mbox{prox}_{g_{i}}=\bigcap_{i=1}^{2}\arg\min g_{i}$, and each $f_{i}$, $h_{i}$ and $T_{i}$ satisfy Assumption \ref{assumption1}. For $x=(x_{1},x_{2},x_{3})^{T}\in H$, we have $\mbox{prox}_{\lambda f_{i}}(x)=P_{D_{i}}(x)$, $\nabla h_{i}(x)=\eta_{i}x-a_{i}$ and notice that $\nabla h_{i}$ is $(1/L_{i})$-Lipschitz continuous with $L_{i}=1$. Moreover, $T_{1}(x)=\mbox{prox}_{\gamma g_{1}}(x) =(I+\gamma A^{T}A)^{-1}(x)$ and $T_{2}(x)=\mbox{prox}_{ g_{2}}(x)=(\mbox{prox}_{ l}(x_{1}),\mbox{prox}_{ l}(x_{2}),\mbox{prox}_{l}(x_{3}))$ 
	where $$\mbox{prox}_{l}(x_{k})=\left\{
	\begin{array}{lr}
	x_{k},& \mbox{if } |x_{k}|<1 \\
	\mbox{sign}(x_{k}),           & \mbox{if }1\leq|x_{k}|\leq2\\
	\mbox{sign}(x_{k}-1),& \mbox{if } |x_{k}|>2.
	\end{array}
	\right.
	$$
	The function $\psi$ is strictly convex since each $h_{i}$ is strictly convex. Moreover, $S=\bigcap_{i=1}^{2}FixT_{i}\neq\emptyset$ and $\Omega=\{\hat{x}\in S: \psi (\hat{x})=\min_{x\in S}\psi (x) \}\neq\emptyset.$ In the experiment we took the bounded, convex and closed subset $X_{i}$ $(i\in\{1,2\})$ of $\mathbb{R}^{3}$ to be a pyramid defined by $X_{i}=\big\{x\in \mathbb{R}^{3}:x_{1}+x_{2}+x_{3}\leq i\big\}\cap\big\{x\in \mathbb{R}^{3}:-x_{1}+x_{2}+x_{3}\leq i\big\}\cap\{x\in \mathbb{R}^{3}:x_{1}-x_{2}+x_{3}\leq i\big\}\cap\big\{x\in \mathbb{R}^{3}:-x_{1}-x_{2}+x_{3}\leq i\big\}\cap\big\{x\in \mathbb{R}^{3}:x_{3}\geq -i\big\}$. Let $D(n)=E(n)+\|w_{n}^{(1)}-w_{n-1}^{(1)}\|+\|w_{n}^{(2)}-w_{n-1}^{(2)}\|$.  
	
	In this experiment, we study the numerical
	behavior of our methods,  Algorithm \ref{algorithm1} and Algorithm \ref{algorithm2},  on test problem (\ref{numer1}) for different choice of parameters and mainly for $\theta_{n}$. The starting points $x_{0}$ and $x_{1}$ in the algorithms are randomly generated and $u^{(1)}=u^{(2)}=w^{(1)}_{0}=w^{(2)}_{0}=z^{(1)}_{0}=z^{(2)}_{0}=z_{0}=x_{0}$. 
	
	\begin{figure}
		\centering
		\includegraphics[width=1\textwidth]{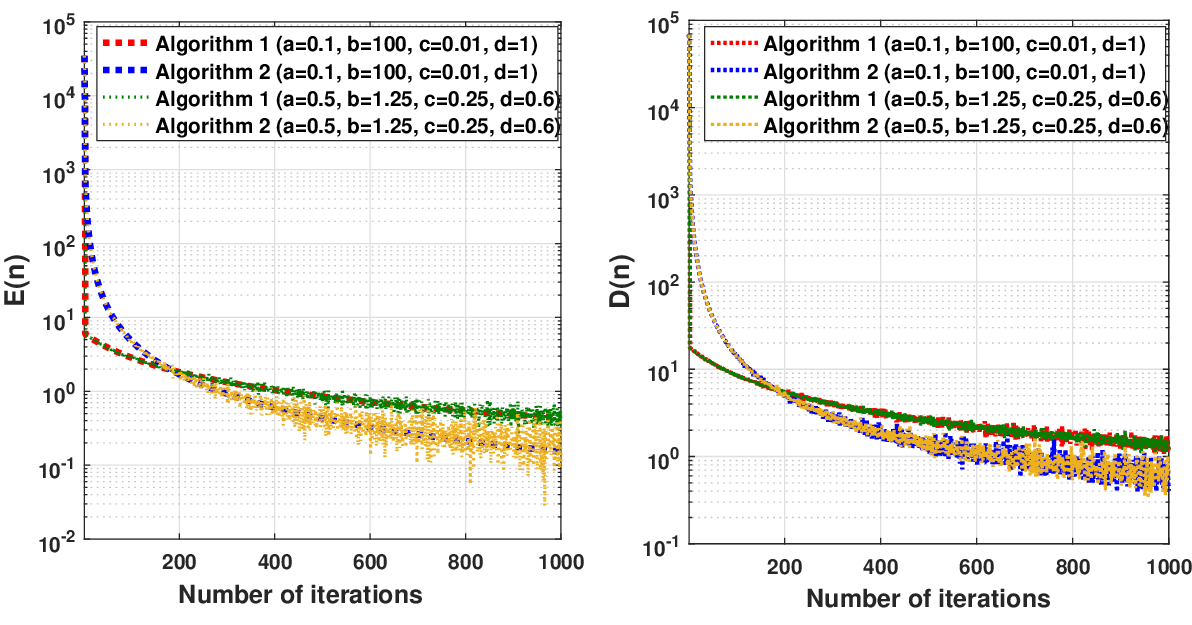}
		\caption{For $\hat{a}=1$, $q=2$,  $\hat{c}=1$, $\hat{d}=2$.}
		\label{Fig 1}
	\end{figure}
	
	\begin{figure}
		\centering
		\includegraphics[width=1\textwidth]{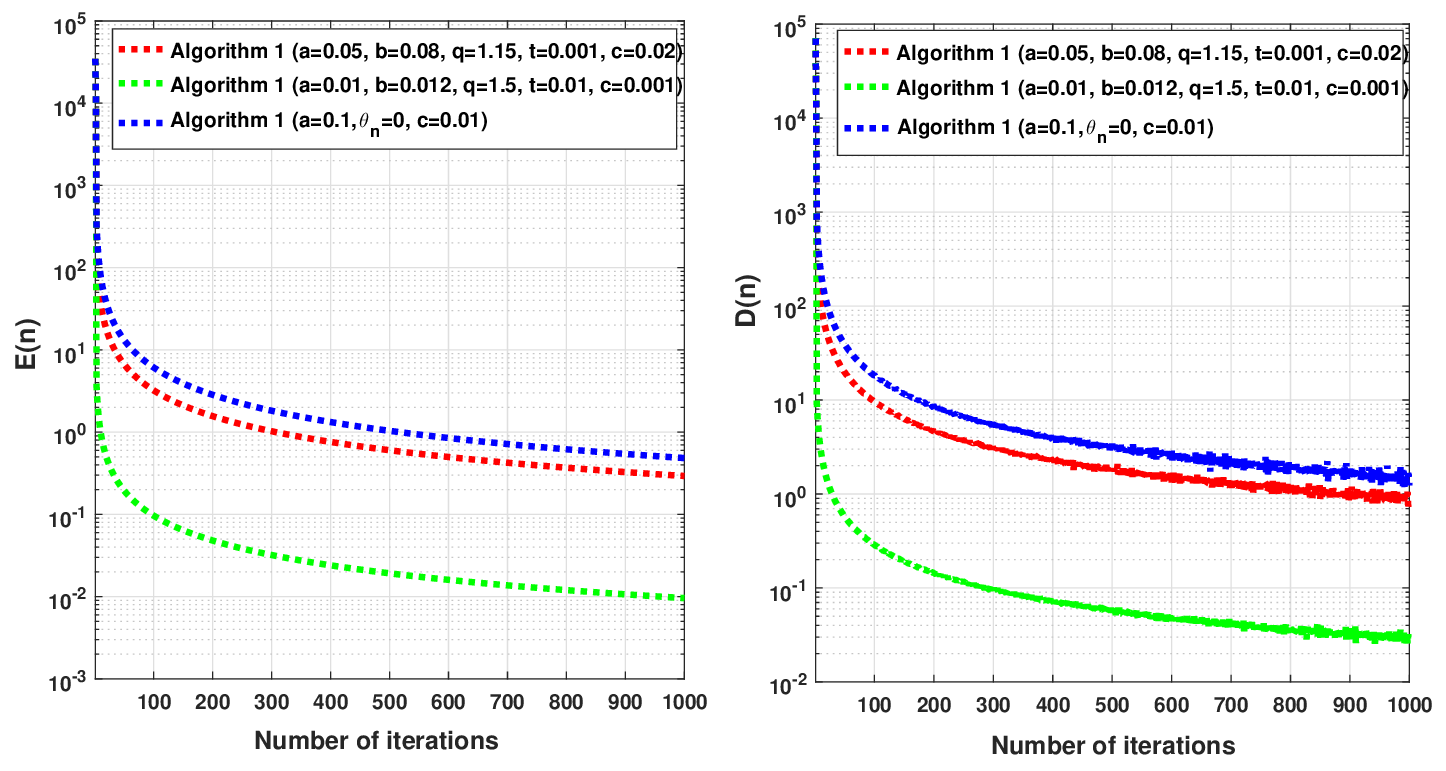}
		\caption{For $\hat{a}=0.1$, $\hat{c}=0.1$, $d=0.5$, $\hat{d}=1.1$.}
		\label{Fig 2}
	\end{figure}
	\begin{figure}
		\centering
		\includegraphics[width=1\textwidth]{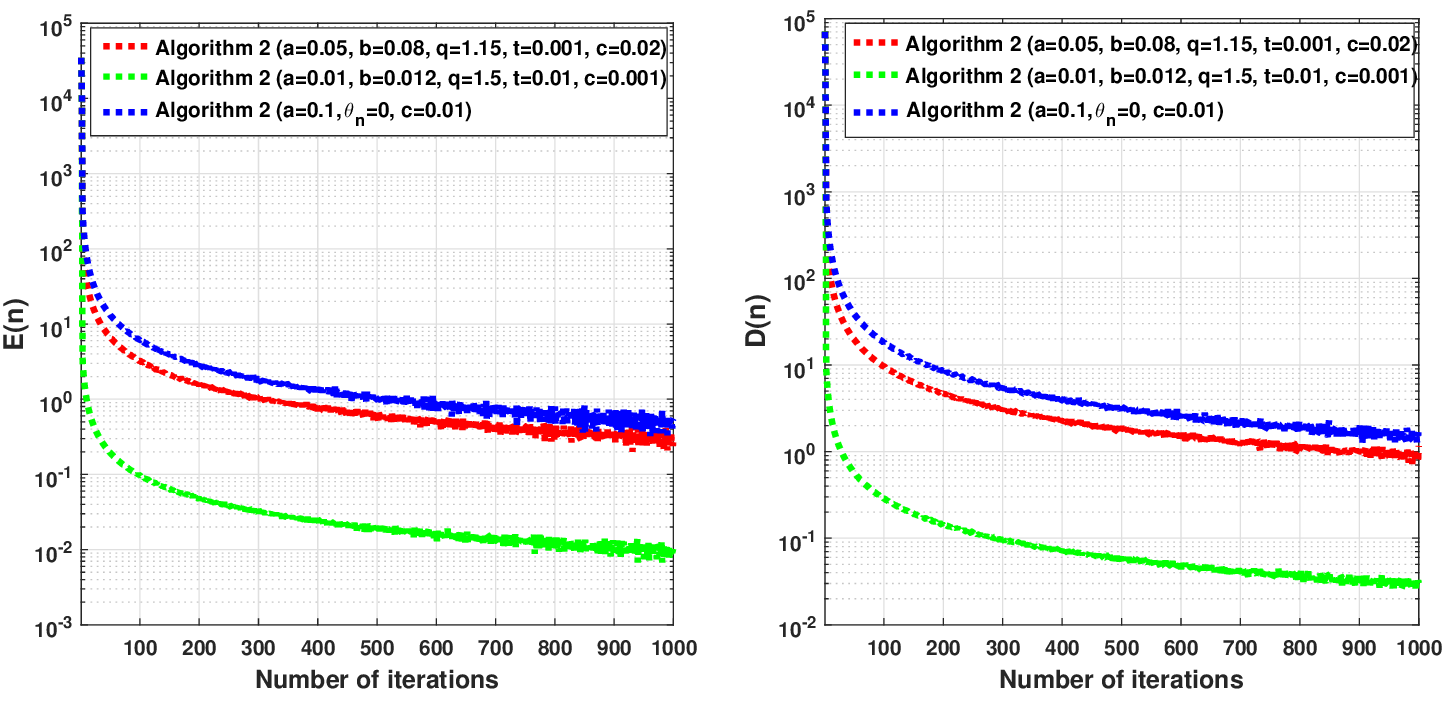}
		\caption{For $\hat{a}=0.1$, $\hat{c}=0.1$, $d=0.5$, $\hat{d}=1.1$.}
		\label{Fig 3}
	\end{figure}
	\begin{table}[]
		\centering
		\caption{For $\theta_{n}=0$ and $\theta_{n}\neq0$ ($t=q=b=2$) where $a=0.5$, $\hat{a}=3$, $c=3$, $\hat{c}=0.01$, $d=3$, $\hat{d}=10$.}\label{Tab 1}
		\begin{tabular}{llllllll}
			\hline
			\multirow{2}{*}{} & \multicolumn{3}{l}{\underline{\hspace{8mm} TOL$=10^{-4}$ \hspace{10mm}}} &  & \multicolumn{3}{l}{\underline{\hspace{8mm} TOL$=10^{-5}$ \hspace{10mm}}}\\ 
			&  $n$   &   CPUt(s)  &   D($n$)  &    & $n$   &   CPUt(s)  &   D($n$)    \\ \hline
			Algorithm 1 ($\theta_{n}\neq0$)	&     244&  0.688580   &  4.353449 &    & 788&   1.152746 &  0.934625\\
			Algorithm 1 ($\theta_{n}=0$)	&   246 &  0.889409  &  4.694379&    & 791 &   2.874446  & 1.060015\\
			Algorithm 2 ($\theta_{n}\neq0$)	&   221  &  0.052896   &  4.600285   &    & 904 &  1.061820  &  0.989179  \\
			Algorithm 2 ($\theta_{n}=0$)	&   221  &   0.092996  &  4.776921   &  & 906 & 2.206435 &       1.023469 \\
			\hline
		\end{tabular}
	\end{table}
	Figures  \ref{Fig 1},  \ref{Fig 2} and \ref{Fig 3} and Table \ref{Tab 1} show that both the proposed algorithms are efficient. From the Figures \ref{Fig 1},  \ref{Fig 2} and \ref{Fig 3}, we can somehow see that the convergence of the algorithms is faster when we pick larger parameters. The results illustrated for several values
	of $\theta_{n}$ in Figure \ref{Fig 2} and Table \ref{Tab 1} shows that the convergence rate for $\theta_{n}=0$ is generally slower than the case $\theta_{n}\neq0$. 
	This highlights the effect and importance the nonzero inertial extrapolation, i.e., $\theta_{n}\neq 0$, in speeding up convergence of the sequence, and it also points that large enough  $\theta_{n}$ value gives a better convergence rate. 
\end{example}
\begin{example}	[\textbf{Comparison to methods in \cite{5,6}}]
	For $\psi:\mathbb{R}^{5}\rightarrow\mathbb{R}$ consider 
	\begin{eqnarray}	\label{numer2}
	\begin{array}{l@{}l}
	&{}\mbox{minimize } \psi(x)=\sum\limits_{k=1}^{8}(\frac{1}{2}x^{T}H_{k}x+v_{k}^{T}x)\\&\hspace{0mm}\mbox{subject to } x\in\bigcap\limits_{i=1}^{4}C_{i},
	\end{array}
	\end{eqnarray}
	where $H_{k}$ ($k\in\{1,\ldots,8\}$) is symmetric and positive definite $5\times 5$ matrix, $v_{k}\in\mathbb{R}^{5}$, $C_{i}=C\bigcap(\bigcap_{j=1}^{i} E^{(j)})$ ($i,j\in I=\{1,2,3,4\}$)  where $C=\big\{x\in \mathbb{R}^{5}:\|x\|\leq 1\big\}$  and  $E^{(j)}$ ($j\in \{1,2,3,4\}$) is a half space given by $E^{(j)}=\big\{x\in \mathbb{R}^{5}:\langle x,d^{(j)}\rangle\leq \zeta^{(j)}\big\}$ for  $d^{(j)}\in\mathbb{R}^{5}$ with $d^{(j)}\neq 0$ and  $\zeta^{(j)}\in\mathbb{R}$.
	
	The objective function $\psi$ in the problem
	(\ref{numer2}) is reduced to
	\begin{equation}	\label{objec}
	\psi(x)=\sum\limits_{i=1}^{4}\Big(\frac{1}{2}x^{T}Q_{i}x+q_{i}^{T}x\Big),
	\end{equation}
	where $Q_{i}=A_{i}+B_{i}$, $q_{i}=a_{i}+b_{i}$,   $A_{i}=H_{2i-1}$, $B_{i}=H_{2i}$, $a_{i}=v_{2i-1}$, and $b_{i}=v_{2i}$. Each matrix $Q_{i}$ is symmetric and positive definite $5\times 5$ matrix. 
	
	Here we compare our methods with the proximal methods in \cite{6} and the methods in \cite{5} in solving (\ref{numer2}). For this purpose the following settings (\textbf{S1} and \textbf{S2}) of (\ref{numer2}) are used. \vspace{2mm}
	\item[\textbf{(i).}] \textbf{S1:} The problem (\ref{numer2}) has a form of MAFPSCOP where  \begin{equation} \label{mapp1} T_{i}=\frac{1}{2}\Big(I+P_{C}\prod_{j=1}^{i}P_{E^{(j)}}\Big),
	\end{equation} $\mathcal{G}_{i}(x)=f_{i}(x)+h_{i}(x)$ for $f_{i}(x)=\frac{1}{2}x^{T}A_{i}x+a^{T}_{i}x$ and $h_{i}(x)=\frac{1}{2}x^{T}B_{i}x+b^{T}_{i}x$. 
	Under this setting $f_{i}$, $h_{i}$, and $T_{i}$ satisfy Assumption \ref{assumption1}. Moreover, for $x\in H$, we have  $\mbox{prox}_{\lambda f_{i}}(x)=(I+A_{i})^{-1}(x-a_{i})$
	and  $\nabla h_{i}(x)=B_{i}x+b_{i}$ and $\nabla h_{i}$ is $\|B_{i}\|=\sqrt{\gamma_{i}}=(1/L_{i})$-Lipschitz  where $\gamma_{i}$ is the spectral radius of $B_{i}$.\\	
	\textbf{S2:} The problem (\ref{numer2}) has a form of MAFPSCOP where $\mathcal{G}_{i}(x)=f_{i}(x)+h_{i}(x)$ and $T_{i}$ is given by (\ref{mapp1}),
	$f_{i}(x)=0$ and $h_{i}(x)=\frac{1}{2}x^{T}Q_{i}x+q^{T}_{i}x$. 
	Note that under this setting each $f_{i}$, $h_{i}$ and $T_{i}$ satisfy Assumption \ref{assumption1}. Moreover, for $x\in H$, we have  $\mbox{prox}_{\lambda f_{i}}(x)=x$,  $\nabla h_{i}(x)=Q_{i}x+q_{i}$ and $\nabla h_{i}$ is $\|Q_{i}\|=\sqrt{\gamma_{i}}=(1/L_{i})$-Lipschitz continuous where $\gamma_{i}$ is the spectral radius of $Q_{i}$.\\ Under the settings \textbf{S1} and \textbf{S2} we test Algorithm \ref{algorithm1} and \ref{algorithm2} for solving (\ref{numer2}) where $a=0.1$, $\hat{a}=100$, $b=10$, $q=101$, $t=0.1$, $c=0.001$, $\hat{c}=100$,
	$d=1$, $\hat{d}=101$.	
	\item[\textbf{(ii).}] The problem (\ref{numer2}) is taken as a form of Problem 2.1 in \cite{6} where $f_{i}(x):=\frac{1}{2}x^{T}Q_{i}x+q_{i}^{T}x$ and $T_{i}$ given by (\ref{mapp1}). Note that, in this case $f_{i}$ and $T_{i}$ satisfy the required conditions in \cite{6}  and $\mbox{prox}_{\lambda f_{i}}(x)=(I+Q_{i})^{-1}(x-q_{i})$. Under this setting of the problem (\ref{numer2}) we test Algorithm 3.1 (PROXALG 3.1) and Algorithm 4.1 (PROXALG 4.1) in \cite{6} where the test parameters of PROXALG 3.1 are $\alpha_{n}=\frac{10^{-1}}{(n+100)^{0.1}}$ and $\gamma_{n}=\frac{10^{-1}}{(n+1)^{0.001}}$, and the test parameters of
	PROXALG 4.1 are $\alpha_{n}=0.5$ and $\gamma_{n}=\frac{10^{-1}}{(n+1)^{0.001}}$.	
	
	\item[\textbf{(iii).}] The problem (\ref{numer2}) is taken as a form of Problem 2.1 in \cite{5} where $f_{i}(x)=\frac{1}{2}x^{T}Q_{i}x+q_{i}^{T}x$  and $T_{i}$ given by (\ref{mapp1}). Note that, in this case $f_{i}$ and $T_{i}$ satisfy the required conditions in \cite{5}, and $\nabla h_{i}(x)=Q_{i}x+q_{i}$ and $\nabla h_{i}$ is $\sqrt{\gamma_{i}}=\frac{1}{L_{i}}$-Lipschitz continuous where $\gamma_{i}$ is the spectral radius of $Q_{i}$. We test Algorithm 3.1 (CGALG 3.1) and Algorithm 4.1 (CGALG 4.1) in \cite{5} for the parameters $\alpha_{n}=\frac{10^{-1}}{(n+1)^{0.5}}$, $\lambda_{n}=\frac{10^{-1}}{(n+1)^{0.25}}$, and $\beta_{n}=\frac{10^{-1}}{n+2}$. 
	
	In the experiments, all required starting points are randomly generated. Moreover, we took $X_{i}=C$, and randomly generated symmetric positive definite $5\times 5$ matrices $A_{i}$ and $B_{i}$ ($Q_{i}=A_{i}+B_{i}$), vectors  $a_{i},b_{i},d^{(j)}\in\mathbb{R}^{5}$ ($q_{i}=a_{i}+b_{i}$), and  real number $\xi^{(j)}$ with
	$Q_{i}v+q_{i}=0$ (or $(I+Q_{i})^{-1}(v-q_{i})=v$) and $\langle v,d^{(j)}\rangle\leq\xi^{(j)}$ where $v=(1/\sqrt{10},\ldots,1/\sqrt{10})^{T}\in\mathbb{R}^{5}$. 
	
	Note that
	$S=\bigcap_{i=1}^{M}FixT_{i}=\bigcap_{i=1}^{M}[C\bigcap\big(\bigcap_{j=1}^{i} E^{(j)}\big)]=\bigcap_{i=1}^{M}C_{i}\neq\emptyset$ and $\Omega=\{\hat{x}\in S: \psi (\hat{x})=\min_{x\in S}\psi (x) \}\neq\emptyset.$  Moreover, $\psi$ is strictly convex function since the each matrix $Q_{i}$ is symmetric positive definite $N\times N$ matrix.  
	
	\begin{figure}
		\centering
		\includegraphics[width=1.05\textwidth]{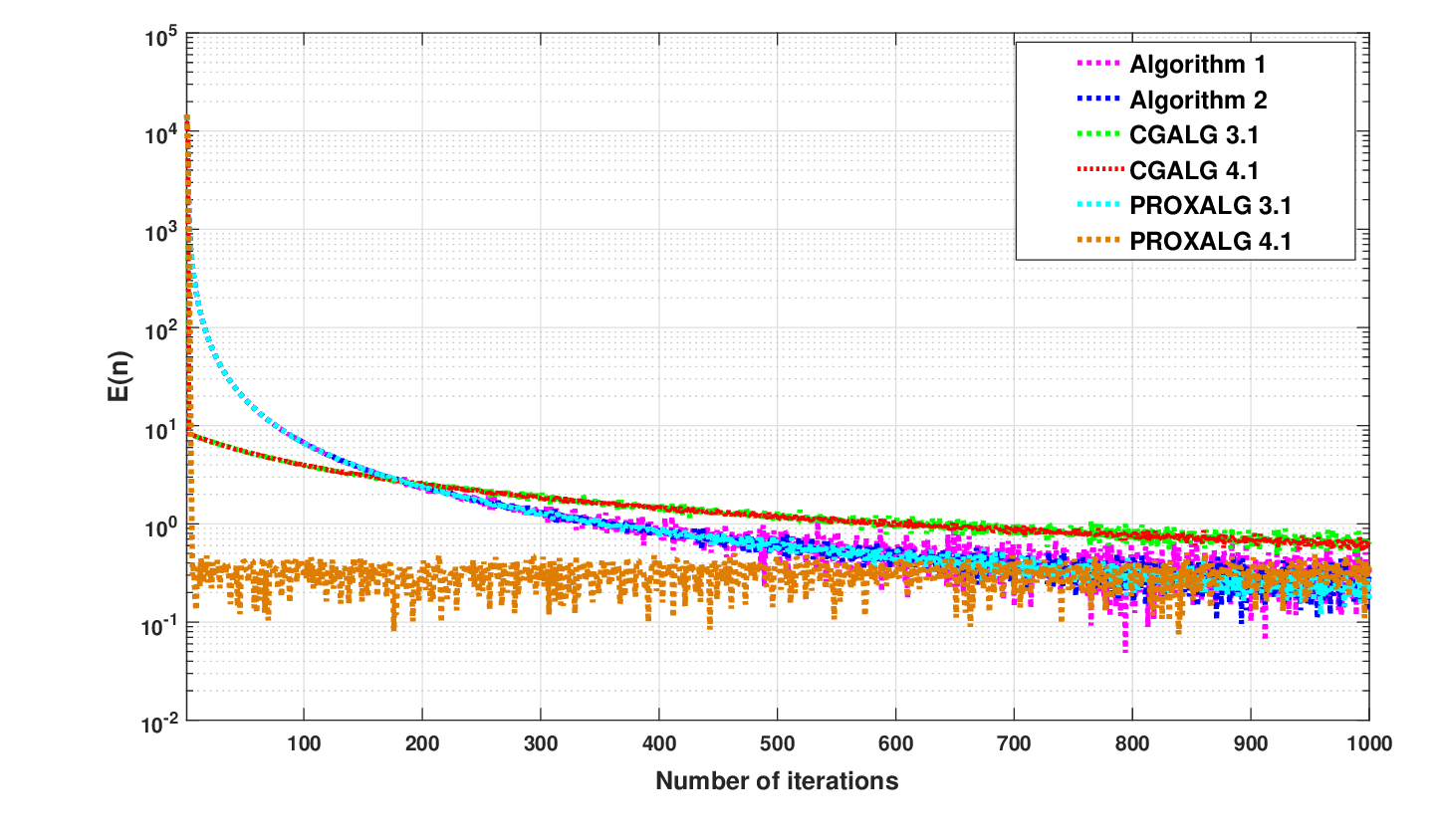}
		\caption{Comparative performance under \textbf{S1}.}
		\label{Fig 4}
	\end{figure}
	\begin{figure}
		\centering
		\includegraphics[width=1.05\textwidth]{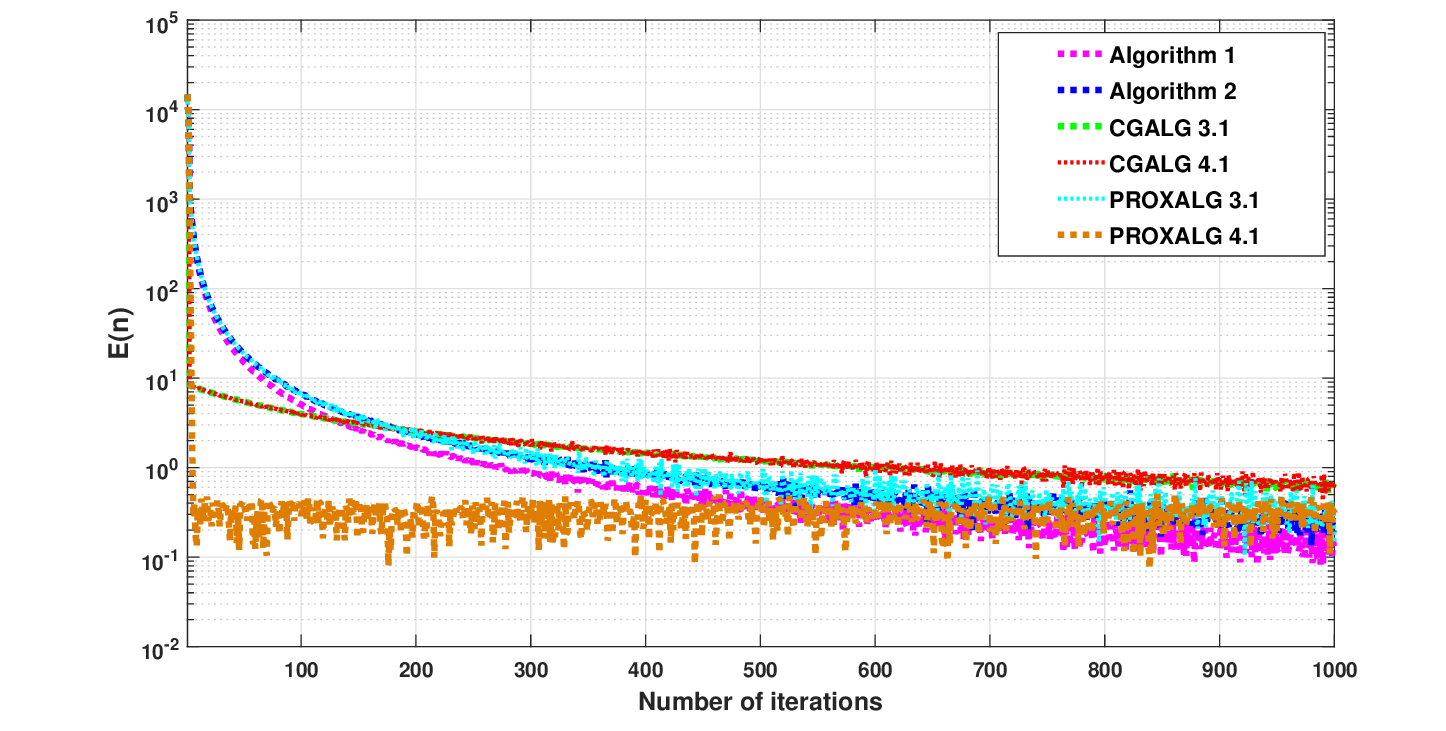}
		\caption{Comparative performance under \textbf{S2}.}
		\label{Fig 5}
	\end{figure}
	\begin{table}[]
		\centering
		\caption{Comparative performance under \textbf{S1} and \textbf{S2}.}\label{Tab 2}
		\begin{tabular}{llllllll}
			\hline
			\multirow{2}{*}{} & \multicolumn{3}{l}{\underline{\hspace{9mm} TOL$=10^{-4}$ \hspace{10mm}}} &  & \multicolumn{3}{l}{\underline{\hspace{9mm} TOL$=10^{-5}$ \hspace{14mm}}} \\ 
			&  $n$   &   CPUt(s)  &   E($n$)  &    & $n$   &   CPUt(s)  &   E($n$)    \\ \hline
			Algorithm 1 (\textbf{S1})  	&  947 &    2.749849   &  0.234959 &  & 1882 &  15.398500   & 0.203173\\
			Algorithm 2 (\textbf{S1})	&  1001  &  3.331085 & 0.637554  &  & 1892 &  17.724907   & 0.086268    \\
			Algorithm 1 (\textbf{S2})     & 964 &    2.171700  &   0.430061& &1887 &  11.40647  & 0.175004	\\
			Algorithm 2 (\textbf{S2})	 &988  &3.243692     &  0.563410  &  & 1904 &  13.330181   & 0.029902  \\
			CGALG 3.1	& 1082  &  2.810116    & 0.825395&&  2177 &   12.096003 & 0.235516    \\
			CGALG 4.1	   & 1211&   2.772201 &  0.532111  & &  \textgreater 2500  &      8.573140&   0.106307 \\
			PROXALG 3.1	 & 1163 &1.969771 &   0.595934 &  &  2253  &     8.072269 &     0.256690  \\
			PROXALG 4.1	    & 996 & 2.369249      & 0.342648 & &  1003 &    10.070528&  0.182512   \\ \hline
		\end{tabular}
	\end{table}
	The experiment evaluations illustrated by Figures \ref{Fig 4} and \ref{Fig 5} and Table \ref{Tab 2} show the computative performance of our  methods. Our algorithms achieve relatively faster convergence than CGALG 3.1, CGALG 4.1, PROXALG 3.1, and PROXALG 4.1. More specifically, from Table \ref{Tab 2}, we can see that Algorithm \ref{algorithm1} takes a minimum of 16 fewer iterations than the others to reach the Tolerance TOL$=10^{-5}$. However, in most cases, our algorithms take more CPU time than CGALG 3.1, CGALG 4.1, PROXALG 3.1, and PROXALG 4.1, and this is not surprising because our algorithms need more computations to accommodate the extended assumptions imposed on the objective function and incorporate more steps to accelerate the convergence of the generated sequences, and this obviously affects the computation time of one iterations.
	
\end{example}
\section*{Conclusions}\label{sec5}

In the paper, we developed a general accelerated iterative method in a distributed setting to solve multi-agent constrained optimization problem. Our approach makes use of proximal, gradient and Halpern methods, and combines inertial extrapolation as an acceleration technique. The numerical experiments show the benefits of inertial extrapolation, the efficiency, and the practical potential of our algorithms. Our next research work is to investigate strongly convergent accelerated algorithms for MAFPSCO and drive the convergence rate of the algorithms.

 \subsection*{Conflict of interest statement} Not Applicable
\bibliography{sn-bibliography}

\section{Appendices}
\section*{APPENDIX A.  Proofs of the
	Lemmas in Subsection \ref{subsec1}}
\subsection*{APPENDIX A.1. Proof of Lemma~\ref{L0}}
\begin{proof}	
If $y^{(i)}_{n}$ in Algorithm \ref{algorithm1} is defined by (\ref{bounded11}), then since $P_{X_{i}}(w)\in X_{i}$ for all $w\in H$ and $X_{i}$ is bounded, it is easy to see that for each user $i$ we have	$\{y_{n}^{(i)}\}\subset X_{i}$, and hence $\{y_{n}^{(i)}\}$ $(i\in I)$ is bounded. Let us show $\{y_{n}^{(i)}\}$ $(i\in I)$ is bounded if $y^{(i)}_{n}$ in Algorithm \ref{algorithm1} is given by (\ref{bounded22}).
Similarly, by definition of projection mapping we have $w_{n}^{(i+1)}=P_{X_{i}}(\alpha_{n}u^{(i)}+(1-\alpha_{n})T_{i}(y^{(i)}_{n}))\in X_{i}$, which implies that $\{w_{n}^{(i+1)}\}$ $(i\in I)$ is a bounded sequence. Since
$x_{n+1}=w^{(M+1)}_{n}$ and $w^{(1)}_{n}=x_{n}$, we have $\{w_{n}^{(i+1)}\}$ $(i\in \{0\}\cup I)$ is bounded as well. Thus, using the definition of $z^{(i)}_{n}$, $\theta_{n}\in [0,1)$ and triangle inequality, we have  $\|z^{(i)}_{n}\|\leq\|w^{(i)}_{n}\|+\theta_{n}\|w^{(i)}_{n}-w^{(i)}_{n-1}\|\leq\|w^{(i)}_{n}\|+\|w^{(i)}_{n}-w^{(i)}_{n-1}\|$ for each user $i\in I$, and thus $\{z_{n}^{(i)}\}$ $(i\in I)$ is bounded sequence.
Since $\{w^{(i)}_{n}\}$ $(i\in I)$ is bounded and $\nabla h_{i}$ $(i\in I)$ is Lipschitz continuous, we have $\{\nabla h_{i}(w^{(i)}_{n})\}$ $(i\in I)$ is bounded sequence. Due to $\lim_{n\rightarrow 0} \beta_{n}=0$ there is $n_{0}\in \mathbb{N}$ such that $\beta_{n}\leq 1/2$ for all $n\geq n_{0}$. Since $\{\nabla h_{i}(w^{(i)}_{n})\}$ $(i\in I)$ is bounded for each $i\in I$ there is $G_{i}>0$ such that 
$\|\nabla h_{i}(w^{(i)}_{n})\|\leq G_{i}(1/2) $. Put $\bar{G}_{i}=\max\Big\{G_{i},\max\limits_{1\leq k\leq n_{0}}\|d^{(i)}_{k}\|\Big\}$. thus, $\|d^{(i)}_{n_{0}}\|\leq \bar{G}_{i}$ $(i\in I)$. Assume that, we have	$\|d^{(i)}_{n}\|\leq \bar{G}_{i}$ for some $n\geq n_{0}$ (induction hypothesis). Then, from
triangle inequality, we get
$$\|d^{(i)}_{n+1}\|\leq\|\nabla h_{i}(w^{(i)}_{n})\|+\beta_{n}\|d^{(i)}_{n}\|\leq\frac{1}{2}G_{i}+\frac{1}{2}\|d^{(i)}_{n}\|\leq\frac{1}{2}G_{i}+\frac{1}{2}\bar{G}_{i}\leq \bar{G}_{i}.$$ Hence, by induction $\|d^{(i)}_{n+1}\|\leq \bar{G}_{i}$ for all $n\geq n_{0}$, which implies $\{d^{(i)}_{n}\}$ $(i\in I)$ is bounded. Let $x\in \arg\min f_{i}$. Since $\arg\min f_{i}=\{y\in H:\mbox{prox}_{\lambda_{n}f_{i}}(y)=y\}=Fix(\mbox{prox}_{\lambda_{n}f_{i}})$ and $\mbox{prox}_{\lambda_{n}f_{i}}$ is also nonexpansive mapping, we get that  
\begin{eqnarray}  \label{eq:100}
\|y^{(i)}_{n}-x\|\nonumber&=&\|\mbox{prox}_{\lambda_{n}f_{i}}(z^{(i)}_{n}+\lambda_{n}d^{(i)}_{n+1})-x\|\nonumber\\&=&\|\mbox{prox}_{\lambda_{n}f_{i}}(z^{(i)}_{n}+\lambda_{n}d^{(i)}_{n+1})-\mbox{prox}_{\lambda_{n}f_{i}}(x)\|\nonumber\\&\leq&\|z^{(i)}_{n}+\lambda_{n}d^{(i)}_{n+1}-x\|\leq\|z^{(i)}_{n}-x\|+2\min_{i\in I}L_{i}\|d^{(i)}_{n+1}\|.
\end{eqnarray}
Therefore, the inequality (\ref{eq:100}) and the boundedness of $\{z^{(i)}_{n}\}$ $(i\in I)$ and $\{d^{(i)}_{n}\}$ $(i\in I)$ implies that $\{y^{(i)}_{n}\}$ $(i\in I)$ is bounded sequence. 
\end{proof}
\subsection*{APPENDIX A.2. Proof of Lemma~\ref{L1}}
\begin{proof} Let $x\in S$. Triangle inequality and the nonexpansiveness of $T_{i}$ gives	
$  
\|T_{i}(y^{(i)}_{n})\|\leq\|T_{i}(y^{(i)}_{n})-T_{i}(x)\|+\|x\|\leq\|y^{(i)}_{n}-x\|+\|x\|$, and this together with the boundedness of $\{y^{(i)}_{n}\}$ $(i\in I)$ from Lemma \ref{L0} implies that $\{T_{i}(y^{(i)}_{n})\}$ $(i\in I)$ is bounded. Since $w^{(i+1)}_{n}=\alpha_{n}u^{(i)}+(1-\alpha_{n})T_{i}(y^{(i)}_{n})$ and $\{T_{i}(y^{(i)}_{n})\}$ $(i\in I)$ is bounded, the sequence $\{w^{(i+1)}_{n}\}$ $(i\in I)$ is bounded. From $x_{n+1}=w^{(M+1)}_{n}$ and $w^{(1)}_{n}=x_{n}$, we get$\{w_{n}^{(i+1)}\}$ $(i\in \{0\}\cup I)$  is bounded, implying that $\{z^{(i)}_{n}\}$ $(i\in I)$ is also  bounded.	Moreover, $\{\nabla h_{i}(w^{(i)}_{n})\}$ $(i\in I)$ is bounded since $\{w^{(i)}_{n}\}$ $(i\in I)$ is bounded and $\nabla h_{i}$ $(i\in I)$ is Lipschitz continuous. Since $\lim_{n\rightarrow 0} \beta_{n}=0$, by similar arguments as in Lemma \ref{L0} there is $\bar{J}_{i}\in\mathbb{R}$ such that $\|d^{(i)}_{n}\|\leq \bar{J}_{i}$ for all $n\in\mathbb{N}$, and hence, $\{d^{(i)}_{n}\}$ $(i\in I)$ is bounded sequence.
\end{proof}
\subsection*{APPENDIX A.3. Proof of Lemma~\ref{L2}}
\begin{proof} $(a)$.  From the definition of $w^{(i+1)}_{n}$ and triangle inequality it follows that 
\begin{eqnarray}  \label{eq:1}
\nonumber&&\|w^{(i+1)}_{n}-w^{(i+1)}_{n-1}\|\nonumber\\&&=\|(\alpha_{n}u^{(i)}+(1-\alpha_{n})T_{i}(y^{(i)}_{n}))-(\alpha_{n-1}u^{(i)}+(1-\alpha_{n-1})T_{i}(y^{(i)}_{n-1}))\|\nonumber\\&&= \|(1-\alpha_{n})(T_{i}(y^{(i)}_{n})-T_{i}(y^{(i)}_{n-1}))+(\alpha_{n}-\alpha_{n-1})(u^{(i)}+T_{i}(y^{(i)}_{n-1}))\|\nonumber\\&&\leq (1-\alpha_{n})\|T_{i}(y^{(i)}_{n})-T_{i}(y^{(i)}_{n-1})\|+|\alpha_{n}-\alpha_{n-1}|\|u^{(i)}+T_{i}(y^{(i)}_{n-1})\|\nonumber\\&&\leq (1-\alpha_{n})\|y^{(i)}_{n}-y^{(i)}_{n-1}\|+|\alpha_{n}-\alpha_{n-1}|\|u^{(i)}+T_{i}(y^{(i)}_{n-1})\|\nonumber\\&&\leq (1-\alpha_{n})\|y^{(i)}_{n}-y^{(i)}_{n-1}\|+|\alpha_{n}-\alpha_{n-1}|\|u^{(i)}+T_{i}(y^{(i)}_{n-1})\|\nonumber\\&&= (1-\alpha_{n})\|y^{(i)}_{n}-y^{(i)}_{n-1}\|+|\alpha_{n}-\alpha_{n-1}|K_{1},
\end{eqnarray}
where $K_{1}=\max_{i\in I}(\sup\{\|u^{(i)}+T_{i}(y^{(i)}_{n-1})\|:n\in\mathbb{N}\})<\infty$. Now let $\bar{y}^{(i)}_{n-1}=\mbox{prox}_{\lambda_{n}f_{i}}(z^{(i)}_{n-1}+\lambda_{n-1}d^{(i)}_{n})$. Then, we also have 
\begin{eqnarray}  \label{eq:2}
\|y^{(i)}_{n}-y^{(i)}_{n-1}\|\leq \|y^{(i)}_{n}-\bar{y}^{(i)}_{n-1}\|+ \|y^{(i)}_{n-1}-\bar{y}^{(i)}_{n-1}\|.
\end{eqnarray}
The nonexpansiveness of proximal, Proposition \ref{P4}, triangle inequality, and Condition \ref{condition1} (decreasing $\lambda_{n}$'s and $\beta_{n}$'s) gives
\begin{eqnarray}  \label{eq:3}
\|y^{(i)}_{n}-\bar{y}^{(i)}_{n-1}\|\nonumber&=&\|\mbox{prox}_{\lambda_{n}f_{i}}(z^{(i)}_{n}+\lambda_{n}d^{(i)}_{n+1})-\mbox{prox}_{\lambda_{n}f_{i}}(z^{(i)}_{n-1}+\lambda_{n-1}d^{(i)}_{n})\| \nonumber\\&\leq&\|z^{(i)}_{n}+\lambda_{n}d^{(i)}_{n+1}-z^{(i)}_{n-1}-\lambda_{n-1}d^{(i)}_{n}\|\nonumber\\&=&\|z^{(i)}_{n}+\lambda_{n}(-\nabla h_{i}(z^{(i)}_{n})+\beta_{n}d^{(i)}_{n})\nonumber\\&&-z^{(i)}_{n-1}-\lambda_{n-1}(-\nabla h_{i}(z^{(i)}_{n-1})+\beta_{n-1}d^{(i)}_{n-1})\|\nonumber\\&=&\|z^{(i)}_{n}-\lambda_{n}\nabla h_{i}(z^{(i)}_{n})-z^{(i)}_{n-1}+\lambda_{n}\nabla h_{i}(z^{(i)}_{n-1}) \nonumber\\&&+(\lambda_{n-1}-\lambda_{n})\nabla h_{i}(z^{(i)}_{n-1})+\lambda_{n}\beta_{n}d^{(i)}_{n}-\lambda_{n-1}\beta_{n-1}d^{(i)}_{n-1}\|\nonumber\\&\leq&\|(z^{(i)}_{n}-\lambda_{n}\nabla h_{i}(z^{(i)}_{n}))-(z^{(i)}_{n-1}-\lambda_{n}\nabla h_{i}(z^{(i)}_{n-1}))\| \nonumber\\&&+|\lambda_{n-1}-\lambda_{n}|\|\nabla h_{i}(z^{(i)}_{n-1})\|+\lambda_{n}\beta_{n}\|d^{(i)}_{n}\|+\lambda_{n-1}\beta_{n-1}\|d^{(i)}_{n-1}\|\nonumber\\&\leq&\|z^{(i)}_{n}-z^{(i)}_{n-1}\| +|\lambda_{n-1}-\lambda_{n}|K_{2}+\lambda_{n-1}\beta_{n-1}K_{3},
\end{eqnarray}
where $K_{2}=\max_{i\in I}(\sup\{\|\nabla h_{i}(z^{(i)}_{n-1})\|:n\in\mathbb{N}\})<\infty$ and $K_{3}=\max_{i\in I}(\sup\{2\|d^{(i)}_{n}\|:n\in\mathbb{N}\})<\infty$. Next we estimate the value of the norm $\|y^{(i)}_{n-1}-\bar{y}^{(i)}_{n-1}\|$ where  $y^{(i)}_{n-1}=\mbox{prox}_{\lambda_{n-1}f_{i}}(z^{(i)}_{n-1}+\lambda_{n-1}d^{(i)}_{n})$ and $\bar{y}^{(i)}_{n-1}=\mbox{prox}_{\lambda_{n}f_{i}}(z^{(i)}_{n-1}+\lambda_{n-1}d^{(i)}_{n})$. Applying Proposition \ref{L2} $(b)$ and on the definitions of $y^{(i)}_{n-1}$ and $\bar{y}^{(i)}_{n-1}$, we have $\frac{z^{(i)}_{n-1}+\lambda_{n-1}d^{(i)}_{n}-y^{(i)}_{n-1}}{\lambda_{n-1}}\in\partial f_{i}(y^{(i)}_{n-1})$ and $\frac{z^{(i)}_{n-1}+\lambda_{n-1}d^{(i)}_{n}-\bar{y}^{(i)}_{n-1}}{\lambda_{n}}\in\partial f_{i}(\bar{y}^{(i)}_{n-1})$. Thanks to the monotonicity property of $\partial f_{i}$, we have
$$
\Big\langle y^{(i)}_{n-1}-\bar{y}^{(i)}_{n-1},\frac{z^{(i)}_{n-1}+\lambda_{n-1}d^{(i)}_{n}-y^{(i)}_{n-1}}{\lambda_{n}}-\frac{z^{(i)}_{n-1}+\lambda_{n-1}d^{(i)}_{n}-\bar{y}^{(i)}_{n-1}}{\lambda_{n-1}}\Big\rangle\geq 0,
$$
which implies by re-arranging it
\begin{eqnarray}\label{eq:4} 
\frac{1}{\lambda_{n-1}\lambda_{n}}\nonumber&&\Big[\langle y^{(i)}_{n-1}-\bar{y}^{(i)}_{n-1},(\lambda_{n-1}-\lambda_{n})(z^{(i)}_{n-1}+\lambda_{n-1}d^{(i)}_{n})\rangle\nonumber\\&&+\langle y^{(i)}_{n-1}-\bar{y}^{(i)}_{n-1},-\lambda_{n-1}(y^{(i)}_{n-1}-\bar{y}^{(i)}_{n-1})\rangle\nonumber\\&&+\langle y^{(i)}_{n-1}-\bar{y}^{(i)}_{n-1},(\lambda_{n}-\lambda_{n-1})\bar{y}^{(i)}_{n-1}\rangle\geq 0.
\end{eqnarray}	
Applying H\'olders inequality in (\ref{eq:4}) yields
\begin{eqnarray} 
\|y^{(i)}_{n-1}-\bar{y}^{(i)}_{n-1}\|^{2}\nonumber&\leq&\frac{|\lambda_{n}-\lambda_{n-1}|}{\lambda_{n-1}}(\|z^{(i)}_{n-1}+\lambda_{n-1}d^{(i)}_{n}\|+\|\bar{y}^{(i)}_{n-1}\|)\|y^{(i)}_{n-1}-\bar{y}^{(i)}_{n-1}\|\nonumber\\&\leq&\frac{|\lambda_{n}-\lambda_{n-1}|}{\lambda_{n-1}}K_{4}\|y^{(i)}_{n-1}-\bar{y}^{(i)}_{n-1}\|,\nonumber
\end{eqnarray}	
where $K_{4}=\max_{i\in I}(\sup\{\|z^{(i)}_{n-1}+\lambda_{n-1}d^{(i)}_{n}\|+\|\bar{y}^{(i)}_{n-1}\|:n\in\mathbb{N}\})<\infty$, and thus
\begin{eqnarray}  \label{eq:5}
\|y^{(i)}_{n-1}-\bar{y}^{(i)}_{n-1}\|&\leq&\frac{|\lambda_{n}-\lambda_{n-1}|}{\lambda_{n-1}}K_{4}.
\end{eqnarray}
Plugging (\ref{eq:3}) and (\ref{eq:5}) into (\ref{eq:2}), we get	
\begin{eqnarray}  \label{eq:6}
\|y^{(i)}_{n}-y^{(i)}_{n-1}\|&\leq&
\|z^{(i)}_{n}-z^{(i)}_{n-1}\| +|\lambda_{n-1}-\lambda_{n}|K_{2}+\lambda_{n-1}\beta_{n-1}K_{3} \\&&\hspace{4mm}+ \frac{|\lambda_{n}-\lambda_{n-1}|}{\lambda_{n-1}}K_{4}.\nonumber
\end{eqnarray}	
Using the definition of $z_{n}^{(i)}$, we get
\begin{eqnarray}  \label{eq:aaa88}
\|z^{(i)}_{n}-z^{(i)}_{n-1}\|\nonumber&\leq&\|w^{(i)}_{n}-w^{(i)}_{n-1}\|+\theta_{n}\|w^{(i)}_{n}-w^{(i)}_{n-1}\|+\theta_{n-1}\|w^{(i)}_{n-1}-w^{(i)}_{n-2}\|\nonumber\\&\leq&\|w^{(i)}_{n}-w^{(i)}_{n-1}\|+\theta_{n-1}K_{5},
\end{eqnarray}
where $K_{5}=\max_{i\in I}(\sup\{2\|w^{(i)}_{n}-w^{(i)}_{n-1}\|:n\in\mathbb{N}\})<\infty$. Consequently, (\ref{eq:1}), (\ref{eq:6}) and (\ref{eq:aaa88}) infer that
$
\|w^{(i+1)}_{n}-w^{(i+1)}_{n-1}\|\leq(1-\alpha_{n})\|w^{(i)}_{n}-w^{(i)}_{n-1}\|+\Lambda_{n},
$
where 	$\Lambda_{n}=\theta_{n-1}K_{5}+|\alpha_{n}-\alpha_{n-1}|K_{1}+|\lambda_{n-1}-\lambda_{n}|K_{2}+\lambda_{n-1}\beta_{n-1}K_{3}+\frac{|\lambda_{n}-\lambda_{n-1}|}{\lambda_{n-1}}K_{4}.$ Applying sum for $i$ from $1$ to $M$ in this inequality, we obtain
$ 
\sum_{i=1}^{M}\big[\|w^{(i+1)}_{n}-w^{(i+1)}_{n-1}\|-\|w^{(i)}_{n}-w^{(i)}_{n-1}\|\big]\leq \sum_{i=1}^{M}\big[-\alpha_{n}\|w^{(i)}_{n}-w^{(i)}_{n-1}\|+\Lambda_{n}\big]\leq -\alpha_{n}\|w^{(1)}_{n}-w^{(1)}_{n-1}\|+M\Lambda_{n}.$ 
This gives 
$ 
\|w^{(M+1)}_{n}-w^{(M+1)}_{n-1}\|-\|w^{(1)}_{n}-w^{(1)}_{n-1}\|\leq -\alpha_{n}\|w^{(1)}_{n}-w^{(1)}_{n-1}\|+M\Lambda_{n},\nonumber
$
and implies by $w^{(M+1)}_{n}=x_{n+1}$ and  $w^{(1)}_{n}=x_{n}$ that	
\begin{eqnarray}  \label{eq:9}
\|x_{n+1}-x_{n}\|&\leq& (1-\alpha_{n})\|x_{n}-x_{n-1}\|+M\Lambda_{n}.
\end{eqnarray}	
From $\lambda_{n-1}\in (0,2\min_{i\in I}L_{i}]$, we get $\frac{|\lambda_{n-1}-\lambda_{n}|}{\lambda_{n}}=2L\frac{|\lambda_{n-1}-\lambda_{n}|}{2L\lambda_{n}}\leq2L\frac{|\lambda_{n-1}-\lambda_{n}|}{\lambda_{n-1}\lambda_{n}}=2L\Big|\frac{1}{\lambda_{n}}-\frac{1}{\lambda_{n-1}}\Big|$. Hence, using (\ref{eq:9}) and the definition of $\Lambda_{n}$, we get 
\begin{eqnarray} \label{eq:aaa10} 
\nonumber&&\frac{\|x_{n+1}-x_{n}\|}{\lambda_{n}}\leq(1-\alpha_{n})\frac{\|x_{n}-x_{n-1}\|}{\lambda_{n}} +\frac{M}{\lambda_{n}}\Lambda_{n}\nonumber\\&&\hspace{5mm}= (1-\alpha_{n})\frac{\|x_{n}-x_{n-1}\|}{\lambda_{n}}+\frac{\theta_{n-1}}{\lambda_{n}}MK_{5}+\frac{|\alpha_{n}-\alpha_{n-1}|}{\lambda_{n}}M K_{1}\nonumber\\&&\hspace{12mm}+\frac{|\lambda_{n-1}-\lambda_{n}|}{\lambda_{n}}MK_{2}+\frac{\lambda_{n-1}}{\lambda_{n}}\beta_{n-1}MK_{3}+\Big|\frac{1}{\lambda_{n}}-\frac{1}{\lambda_{n-1}}\Big|MK_{4}\nonumber\\&&\hspace{5mm}\leq (1-\alpha_{n})\frac{\|x_{n}-x_{n-1}\|}{\lambda_{n}}+\frac{\theta_{n-1}}{\lambda_{n}}MK_{5}+\frac{|\alpha_{n}-\alpha_{n-1}|}{\lambda_{n}}M K_{1}\nonumber\\&&\hspace{12mm}+\frac{\lambda_{n-1}}{\lambda_{n}}\beta_{n-1}MK_{3}+\Big|\frac{1}{\lambda_{n}}-\frac{1}{\lambda_{n-1}}\Big|(2LMK_{2}+MK_{4})\nonumber\\&&\hspace{5mm}\leq (1-\alpha_{n})\frac{\|x_{n}-x_{n-1}\|}{\lambda_{n-1}}+(1-\alpha_{n})\Big|\frac{\|x_{n}-x_{n-1}\|}{\lambda_{n}}-\frac{\|x_{n}-x_{n-1}\|}{\lambda_{n-1}}\Big|\nonumber\\&&\hspace{12mm}+\frac{\theta_{n-1}}{\lambda_{n}}MK_{5}+\frac{|\alpha_{n}-\alpha_{n-1}|}{\lambda_{n}}M K_{1}+\sigma\beta_{n-1}MK_{3}\nonumber\\&&\hspace{12mm}+\Big|\frac{1}{\lambda_{n}}-\frac{1}{\lambda_{n-1}}\Big|(2LMK_{2}+MK_{4})\nonumber\\&&\hspace{5mm}\leq (1-\alpha_{n})\frac{\|x_{n}-x_{n-1}\|}{\lambda_{n-1}}+\Big|\frac{1}{\lambda_{n}}-\frac{1}{\lambda_{n-1}}\Big|K_{6}+\frac{\theta_{n-1}}{\lambda_{n}}MK_{5}\nonumber\\&&\hspace{12mm}+\frac{|\alpha_{n}-\alpha_{n-1}|}{\lambda_{n}}M K_{1}+\sigma\beta_{n-1}MK_{3}+\Big|\frac{1}{\lambda_{n}}-\frac{1}{\lambda_{n-1}}\Big|(2LMK_{2}+MK_{4})\nonumber\\&&\hspace{5mm}= (1-\alpha_{n})\frac{\|x_{n}-x_{n-1}\|}{\lambda_{n-1}}+\frac{|\alpha_{n}-\alpha_{n-1}|}{\lambda_{n}}M K_{1}+\frac{\theta_{n-1}}{\lambda_{n}}MK_{5}\nonumber\\&&\hspace{12mm}+\sigma\beta_{n-1}MK_{3}+\Big|\frac{1}{\lambda_{n}}-\frac{1}{\lambda_{n-1}}\Big|K_{7},
\end{eqnarray}	
where $K_{6}=\sup\{\|x_{n}-x_{n-1}\|:n\in\mathbb{N}\}<\infty$ and $K_{7}=K_{6}+2LMK_{2}+MK_{4}$.  Applying (\ref{eq:aaa10}), we have
\begin{eqnarray}  \label{eq:10}
\frac{\|x_{n+1}-x_{n}\|}{\lambda_{n}}\leq (1-\alpha_{n})\frac{\|x_{n}-x_{n-1}\|}{\lambda_{n-1}}+\alpha_{n}\psi_{n},
\end{eqnarray}
where
\begin{eqnarray}  
\psi_{n}\nonumber=\frac{1}{\lambda_{n}}\Big|1-\frac{\alpha_{n-1}}{\alpha_{n}}\Big|M K_{1}+\frac{\theta_{n-1}}{\lambda_{n}\alpha_{n}}MK_{5}+\sigma\frac{\beta_{n-1}}{\alpha_{n}}KM_{3}+\frac{1}{\alpha_{n}}\Big|\frac{1}{\lambda_{n}}-\frac{1}{\lambda_{n-1}}\Big|K_{7}.
\end{eqnarray}	
Therefore, the result in (\ref{eq:10}) and the parameter restrictions in Assumption 2 in view of Proposition \ref{P5} guarantees
$
\lim_{n\rightarrow\infty}\frac{\|x_{n+1}-x_{n}\|}{\lambda_{n}}=0,
$
which with $\lim_{n\rightarrow\infty}\lambda_{n}=0$ implies that
$
\lim_{n\rightarrow\infty}\|x_{n+1}-x_{n}\|=0.
$
\item{$(b)$.} From $y^{(i)}_{n}=\mbox{prox}_{\lambda_{n}f_{i}}(z^{(i)}_{n}+\lambda_{n}d^{(i)}_{n+1})$ and Proposition \ref{P2} $(b)$, we have $z^{(i)}_{n}+\lambda_{n}d^{(i)}_{n+1}-y^{(i)}_{n}\in\lambda_{n}\partial f_{i}(y^{(i)}_{n})$ and hence by definition of subdifferential, we get
$\langle   z^{(i)}_{n}-y^{(i)}_{n}+\lambda_{n}d^{(i)}_{n+1},\bar{x}-y^{(i)}_{n}\rangle \leq \lambda_{n}(f_{i}(\bar{x})-f_{i}(y^{(i)}_{n})).$ Thus, by applying the equality $
2\langle  z^{(i)}_{n}-y^{(i)}_{n},\bar{x}-y^{(i)}_{n}\rangle =\|y^{(i)}_{n}-\bar{x}\|^{2}+\|z^{(i)}_{n}-y^{(i)}_{n}\|^{2}-\|z^{(i)}_{n}-\bar{x}\|^{2}$, it follows that
\begin{eqnarray}  \label{eq:14}
\|y^{(i)}_{n}-\bar{x}\|^{2}\nonumber&\leq&\|z^{(i)}_{n}-\bar{x}\|^{2}-\|z^{(i)}_{n}-y^{(i)}_{n}\|^{2}+2\lambda_{n}(f_{i}(\bar{x})-f_{i}(y^{(i)}_{n}))\nonumber\\&&+2\lambda_{n}\langle  d^{(i)}_{n+1},y^{(i)}_{n}-\bar{x}\rangle.
\end{eqnarray}
The definitions of $z^{(i)}_{n}$ and the triangle inequality ensures
\begin{eqnarray}  \label{eq:15}
\|z^{(i)}_{n}-\bar{x}\|^{2}\leq(\|w^{(i)}_{n}-\bar{x}\|+\theta_{n}\|w^{(i)}_{n}-w^{(i)}_{n-1}\|)^{2}\leq\|w^{(i)}_{n}-\bar{x}\|^{2}+\theta_{n}Q_{1},
\end{eqnarray}
where $Q_{1}=\max_{i\in I}(\sup\{\|w^{(i)}_{n}-w^{(i)}_{n-1}\|^{2}+2\|w^{(i)}_{n}-w^{(i)}_{n-1}\|\|w^{(i)}_{n}-\bar{x}\|:n\in\mathbb{N}\})<\infty$. Moreover, the definitions of $w^{(i+1)}_{n}$, the convexity of $\|.\|^{2}$, $\alpha_{n}\in(0,1]$, and firmly nonexpansiveness of $T_{i}$ gives
\begin{eqnarray}  \label{eq:16}
\|w^{(i+1)}_{n}-\bar{x}\|^{2}\nonumber&\leq&(1-\alpha_{n})\|T_{i}(y^{(i)}_{n})-\bar{x}\|^{2}+\alpha_{n}\|u^{(i)}-\bar{x}\|^{2}\nonumber\\&\leq&\|y^{(i)}_{n}-\bar{x}\|^{2}-(1-\alpha_{n})\|T_{i}(y^{(i)}_{n})-y^{(i)}_{n}\|^{2}+\alpha_{n}Q_{2},
\end{eqnarray}
where $Q_{2}=\max\{\|u^{(i)}-\bar{x}\|^{2}:i\in I\}<\infty$. Hence, (\ref{eq:14}), (\ref{eq:15}) and (\ref{eq:16}) yields
\begin{eqnarray}  
\|w^{(i+1)}_{n}-\bar{x}\|^{2}\nonumber&\leq&\|w^{(i)}_{n}-\bar{x}\|^{2}+\theta_{n}Q_{1}+\alpha_{n}Q_{2}-(1-\alpha_{n})\|T_{i}(y^{(i)}_{n})-y^{(i)}_{n}\|^{2}\nonumber\\&&-\|z^{(i)}_{n}-y^{(i)}_{n}\|^{2}+2\lambda_{n}(f_{i}(\bar{x})-f_{i}(y^{(i)}_{n}))+2\lambda_{n}\langle  d^{(i)}_{n+1},y^{(i)}_{n}-\bar{x}\rangle.\nonumber 
\end{eqnarray}

\end{proof}
\subsection*{APPENDIX A.4. Proof of Lemma~\ref{L3}}
\begin{proof} $(a).$ Let $\bar{x}\in S$. Now, in view of	Proposition \ref{P2} $(d)$, there exists
$q^{(i)}\in\partial f_{i}(\bar{x})$  and $\partial f_{i}(\bar{x})$ ($i\in I$) is  bounded. Thus, by the
definition of $\partial f_{i}$ and the boundedness $\{y_{n}^{(i)}\}$ ($i\in I$) from Lemma \ref{L1}, there is a nonnegative real number $Q_{0}$ such that
$2(f_{i}(\bar{x})-f_{i}(y_{n}^{(i)}))\leq 2 \langle \bar{x}-y_{n}^{(i)},q^{(i)}\rangle\leq 2\|\bar{x}-y_{n}^{(i)}\|\|q^{(i)}\|\leq Q_{0},$
for all $n\in\mathbb{N}$, $i\in I$.  
Thus, using Lemma \ref{L2}, we have
\begin{eqnarray}
\|w^{(i+1)}_{n}-\bar{x}\|^{2}&\leq&\|w^{(i)}_{n}-\bar{x}\|^{2}+\theta_{n}Q_{1}+\alpha_{n}Q_{2}+\lambda_{n}Q_{0}+\lambda_{n}Q_{3}\label{eq:17}\\&&\hspace{6mm}-(1-\alpha_{n})\|T_{i}(y^{(i)}_{n})-y^{(i)}_{n}\|^{2}-\|z^{(i)}_{n}-y^{(i)}_{n}\|^{2},\nonumber
\end{eqnarray}
where
$Q_{3}=\max_{i\in I}(\sup\{2\langle  d^{(i)}_{n+1},y^{(i)}_{n}-\bar{x}\rangle:n\in\mathbb{N}\})<\infty$. 
Therefore, summing both sides of the inequality (\ref{eq:17}) from $i = 1$ to $M$, we obtain  
$\|w^{(M+1)}_{n}-\bar{x}\|^{2}\leq\|w^{(1)}_{n}-\bar{x}\|^{2}+M\Upsilon_{n}-(1-\alpha_{n})\sum_{i=1}^{M}\|T_{i}(y^{(i)}_{n})-y^{(i)}_{n}\|^{2}-\sum_{i=1}^{M}\|z^{(i)}_{n}-y^{(i)}_{n}\|^{2},$
where $\Upsilon_{n}=\theta_{n}Q_{1}+\alpha_{n}Q_{2}+\lambda_{n}(Q_{0}+Q_{3})$. Noting $w^{(M+1)}_{n}=x_{n+1}$ and $w^{(1)}_{n}=x_{n}$, we get
$\|x_{n+1}-\bar{x}\|^{2}\leq\|x_{n}-\bar{x}\|^{2}+M\Upsilon_{n}-(1-\alpha_{n})\sum_{i=1}^{M}\|T_{i}(y^{(i)}_{n})-y^{(i)}_{n}\|^{2}-\sum_{i=1}^{M}\|z^{(i)}_{n}-y^{(i)}_{n}\|^{2},
$
which leads to
\begin{eqnarray}\label{eq:18}\nonumber&&(1-\alpha_{n})\sum\limits_{i=1}^{M}\|T_{i}(y^{(i)}_{n})-y^{(i)}_{n}\|^{2}+\sum\limits_{i=1}^{M}\|z^{(i)}_{n}-y^{(i)}_{n}\|^{2}\nonumber\\&&\hspace{14mm}\leq\|x_{n}-\bar{x}\|^{2}-\|x_{n+1}-\bar{x}\|^{2}+M\Upsilon_{n}\nonumber\\&&\hspace{14mm}=(\|x_{n}-\bar{x}\|-\|x_{n+1}-\bar{x}\|)(\|x_{n}-\bar{x}\|+\|x_{n+1}-\bar{x}\|)+M\Upsilon_{n}\nonumber\\&&\hspace{14mm}\leq\|x_{n}-x_{n+1}\|Q_{4}+M\Upsilon_{n},
\end{eqnarray}
where  
$Q_{4}=\sup\{\|x_{n}-\bar{x}\|+\|x_{n+1}-\bar{x}\|:n\in\mathbb{N}\}<\infty$.
From $\theta_{n}, \alpha_{n}, \lambda_{n}\rightarrow 0$ in Condition \ref{condition1}, we have $\Upsilon_{n}=\theta_{n}Q_{1}+\alpha_{n}Q_{2}+\lambda_{n}(Q_{0}+Q_{3})\rightarrow 0$,
and by Lemma \ref{L2} $(a)$, we have $\|x_{n}-x_{n+1}\|\rightarrow0$. Hence, from (\ref{eq:18}),  we obtain
$
\lim\limits_{n\rightarrow\infty}\big[(1-\alpha_{n})\sum_{i=1}^{M}\|T_{i}(y^{(i)}_{n})-y^{(i)}_{n}\|^{2}+\sum_{i=1}^{M}\|z^{(i)}_{n}-y^{(i)}_{n}\|^{2}\big]=0
$
implying that $
\lim\limits_{n\rightarrow\infty}\big[\sum_{i=1}^{M}\|T_{i}(y^{(i)}_{n})-y^{(i)}_{n}\|^{2}+\sum_{i=1}^{M}\|z^{(i)}_{n}-y^{(i)}_{n}\|^{2}\big]=0,
$ and hence
\begin{eqnarray}\label{eq:19}
\lim\limits_{n\rightarrow\infty}\|z^{(i)}_{n}-y^{(i)}_{n}\|=\lim_{n\rightarrow\infty}\|T_{i}(y^{(i)}_{n})-y^{(i)}_{n}\|=0.
\end{eqnarray}
On the other hand, using the definition of $z^{(i)}_{n}$ and $w_{n}^{(i+1)}$ in the algorithm, we get $\|w_{n}^{(i)}-z^{(i)}_{n}\|=\theta_{n}\|w_{n}^{(i)}-w_{n-1}^{(i)}\|$ and $\|w_{n}^{(i+1)}-T_{i}(y^{(i)}_{n})\|=\alpha_{n}\|u^{(i)}-T_{i}(y^{(i)}_{n})\|.$
Thus, the boundedness of $\{w_{n}^{(i)}\}$ ($i\in I$) and $\{T_{i}(y^{(i)}_{n})\}$ ($i\in I$) together with $ \lim_{n\rightarrow\infty}\alpha_{n}=\lim_{n\rightarrow\infty}\theta_{n}= 0$ gives
\begin{eqnarray}\label{eq:21}
\lim\limits_{n\rightarrow\infty}\|w_{n}^{(i)}-z^{(i)}_{n}\|= \lim\limits_{n\rightarrow\infty}\|w_{n}^{(i+1)}-T_{i}(y^{(i)}_{n})\|=0.
\end{eqnarray}
Combining (\ref{eq:19}) and (\ref{eq:21}), we obtain
\begin{eqnarray}\label{eq:22}
\lim\limits_{n\rightarrow\infty}\|y_{n}^{(i)}-w^{(i+1)}_{n}\|=\lim\limits_{n\rightarrow\infty}\|w_{n}^{(i)}-y^{(i)}_{n}\|=0.
\end{eqnarray}
Moreover, by (\ref{eq:19}) and (\ref{eq:22}), we get $
\lim_{n\rightarrow\infty}\|z_{n}^{(i)}-w^{(i+1)}_{n}\|=0.$
Noting $x_{n}=w^{(1)}_{n}$, for each $i=2,\ldots,M$ we have  
$
\|x_{n}-w^{(i)}_{n}\|=\sum_{k=1}^{i-1}\|w^{(k)}_{n}-w^{(k+1)}_{n}\|\leq\sum_{k=1}^{i-1}(\|w^{(k)}_{n}-z^{(k)}_{n}\|+\|z^{(k)}_{n}-w^{(k+1)}_{n}\|),
$
and hence this together with (\ref{eq:21}) and $\lim_{n\rightarrow\infty}\|z_{n}^{(i)}-w^{(i+1)}_{n}\|=0$ gives
$\lim_{n\rightarrow\infty}\|x_{n}-w^{(i)}_{n}\|=0$. Moreover, (\ref{eq:22}) and $\lim_{n\rightarrow\infty}\|x_{n}-w^{(i)}_{n}\|=0$ together with the triangle inequality yields 
\begin{eqnarray}\label{eq:24}\lim\limits_{n\rightarrow\infty}\|x_{n}-y^{(i)}_{n}\|=0.\end{eqnarray}
\item{$(b)$.} From (\ref{eq:19}), (\ref{eq:24}), and
$\|x_{n}-z^{(i)}_{n}\|\leq\|x_{n}-y^{(i)}_{n}\|+\|y^{(i)}_{n}-z^{(i)}_{n}\|$, we get $\lim_{n\rightarrow\infty}\|x_{n}-z^{(i)}_{n}\|=0$. Now, noting that $T_{i}$ is  nonexpansive, we obtain
\begin{eqnarray}\label{eq:24jjj}
\|x_{n}-T_{i}(x_{n})\|\nonumber&=&\|x_{n}-y^{(i)}_{n}\|+\|y^{(i)}_{n}-T_{i}(y^{(i)}_{n})\|+\|T_{i}(x_{n})-T_{i}(y^{(i)}_{n})\|\nonumber\\&\leq&2\|x_{n}-y^{(i)}_{n}\|+\|y^{(i)}_{n}-T_{i}(y^{(i)}_{n})\|.
\end{eqnarray}
Hence,  (\ref{eq:19}), (\ref{eq:24}) and  (\ref{eq:24jjj}) imply 
$
\lim_{n\rightarrow\infty}\|x_{n}-T_{i}(x_{n})\|=0
$ for all $i\in I$.
\end{proof}
\section*{APPENDIX B.  Proofs of the
Lemmas in Subsection \ref{subsec2}}
\subsection*{APPENDIX B.1. Proof of Lemma~\ref{L4}}
\begin{proof}
We omit the proof since it
is similar to the proof of Lemma \ref{L1}.	
\end{proof}
\subsection*{APPENDIX B.2. Proof of Lemma~\ref{L5}}
\begin{proof}
\item{$(a)$.} Followed from the definition of $w^{(i+1)}_{n}$ and the triangle inequality, we get
\begin{eqnarray}  \label{eq:bb1}
\|w^{(i+1)}_{n}-w^{(i+1)}_{n-1}\|\leq (1-\alpha_{n})\|y^{(i)}_{n}-y^{(i)}_{n-1}\|+|\alpha_{n}-\alpha_{n-1}|W_{1},
\end{eqnarray}
where $W_{1}=\max_{i\in I}(\sup\{\|u^{(i)}+T_{i}(y^{(i)}_{n-1})\|:n\in\mathbb{N}\})<\infty$. Note that  
\begin{eqnarray}  \label{eq:bb2}
\|y^{(i)}_{n}-y^{(i)}_{n-1}\|\leq \|y^{(i)}_{n}-\bar{y}^{(i)}_{n-1}\|+ \|y^{(i)}_{n-1}-\bar{y}^{(i)}_{n-1}\|.
\end{eqnarray}
where $\bar{y}^{(i)}_{n-1}=\mbox{prox}_{\lambda_{n}f_{i}}(z_{n-1}+\lambda_{n-1}d^{(i)}_{n})$. Using the fact that $\mbox{prox}_{\lambda_{n}f_{i}}$ and $I-\lambda_{n}\nabla h_{i}$ are nonexpansive, it follows that
\begin{eqnarray}  \label{eq:bb3}
\|y^{(i)}_{n}-\bar{y}^{(i)}_{n-1}\|\nonumber&=&\|\mbox{prox}_{\lambda_{n}f_{i}}(z_{n}+\lambda_{n}d^{(i)}_{n+1})-\mbox{prox}_{\lambda_{n}f_{i}}(z_{n-1}+\lambda_{n-1}d^{(i)}_{n})\| \nonumber\\&\leq&\|z_{n}+\lambda_{n}d^{(i)}_{n+1}-z_{n-1}-\lambda_{n-1}d^{(i)}_{n}\|\nonumber\\&=&\|z_{n}+\lambda_{n}(-\nabla h_{i}(z_{n})+\beta_{n}d^{(i)}_{n})\nonumber\\&&-z_{n-1}-\lambda_{n-1}(-\nabla h_{i}(z_{n-1})+\beta_{n-1}d^{(i)}_{n-1})\|\nonumber\\&\leq&\|(z_{n}-\lambda_{n}\nabla h_{i}(z_{n}))-(z_{n-1}-\lambda_{n}\nabla h_{i}(z_{n-1}))\| \nonumber\\&&+|\lambda_{n-1}-\lambda_{n}|\|\nabla h_{i}(z_{n-1})\|+\lambda_{n}\beta_{n}\|d^{(i)}_{n}\|+\lambda_{n-1}\beta_{n-1}\|d^{(i)}_{n-1}\|\nonumber\\&\leq&\|z_{n}-z_{n-1}\| +|\lambda_{n-1}-\lambda_{n}|W_{2}+\lambda_{n-1}\beta_{n-1}W_{3}
\end{eqnarray}
for each $i\in I$, where $W_{2}=\max\limits_{i\in I}(\sup\{\|\nabla h_{i}(z_{n-1})\|:n\in\mathbb{N}\})<\infty$ and $W_{3}=\max\limits_{i\in I}(\sup\{2\|d^{(i)}_{n}\|:n\in\mathbb{N}\})<\infty$. The definitions of  $y^{(i)}_{n-1}$ and $\bar{y}^{(i)}_{n-1}$ gives $\frac{z_{n-1}+\lambda_{n-1}d^{(i)}_{n}-y^{(i)}_{n-1}}{\lambda_{n-1}}\in\partial f_{i}(y^{(i)}_{n-1})$ and $\frac{z_{n-1}+\lambda_{n-1}d^{(i)}_{n}-\bar{y}^{(i)}_{n-1}}{\lambda_{n}}\in\partial f_{i}(\bar{y}^{(i)}_{n-1})$. Applying the monotonicity of $\partial f_{i}$, we get $$
\big\langle y^{(i)}_{n-1}-\bar{y}^{(i)}_{n-1},\frac{z_{n-1}+\lambda_{n-1}d^{(i)}_{n}-y^{(i)}_{n-1}}{\lambda_{n}}-\frac{z_{n-1}+\lambda_{n-1}d^{(i)}_{n}-\bar{y}^{(i)}_{n-1}}{\lambda_{n-1}}\big\rangle\geq 0,
$$
and rearranging it and using triangle inequality, we obtain
\begin{eqnarray}  
\|y^{(i)}_{n-1}-\bar{y}^{(i)}_{n-1}\|^{2}\nonumber&\leq&\frac{|\lambda_{n}-\lambda_{n-1}|}{\lambda_{n-1}}(\|z_{n-1}+\lambda_{n-1}d^{(i)}_{n}\|+\|\bar{y}^{(i)}_{n-1}\|)\|y^{(i)}_{n-1}-\bar{y}^{(i)}_{n-1}\|\nonumber\\&\leq&\frac{|\lambda_{n}-\lambda_{n-1}|}{\lambda_{n-1}}W_{4}\|y^{(i)}_{n-1}-\bar{y}^{(i)}_{n-1}\|,\nonumber
\end{eqnarray}	
where $W_{4}=\max_{i\in I}(\sup\{\|z_{n-1}+\lambda_{n-1}d^{(i)}_{n}\|+\|\bar{y}^{(i)}_{n-1}\|:n\in\mathbb{N}\})<\infty$. That is, 
\begin{eqnarray}  \label{eq:bb5}
\|y^{(i)}_{n-1}-\bar{y}^{(i)}_{n-1}\|\leq\frac{|\lambda_{n}-\lambda_{n-1}|}{\lambda_{n-1}}W_{4}.
\end{eqnarray}
Using the definition of $z_{n}$, we have $\|z_{n}-z_{n-1}\|\leq\|x_{n}-x_{n-1}\|+\theta_{n-1}W_{5},$
where $W_{5}=\max_{i\in I}(\sup\{2\|x_{n}-x_{n-1}\|:n\in\mathbb{N}\})<\infty$, and this together with the inequalities (\ref{eq:bb1})-(\ref{eq:bb5}) gives
\begin{eqnarray}  \label{eq:bb7}
&&\|w^{(i+1)}_{n}-w^{(i+1)}_{n-1}\|\leq (1-\alpha_{n})\|x_{n}-x_{n-1}\|+\theta_{n-1}W_{5}+|\alpha_{n}-\alpha_{n-1}|W_{1}\nonumber\\&&\hspace{20mm}+|\lambda_{n-1}-\lambda_{n}|W_{2}+\lambda_{n-1}\beta_{n-1}W_{3}+\frac{|\lambda_{n}-\lambda_{n-1}|}{\lambda_{n-1}}W_{4}.
\end{eqnarray}
From the definition of $\{x_{n+1}\}$ and (\ref{eq:bb7}), we obtain
\begin{eqnarray}  
\|x_{n+1}-x_{n-1}\|\nonumber&=&\Big\|\frac{1}{M}\sum\limits_{i=1}^{M}w^{(i+1)}_{n}-\frac{1}{M}\sum\limits_{i=1}^{M}w^{(i+1)}_{n-1}\Big\|\leq\frac{1}{M}\sum\limits_{i=1}^{M}\|w^{(i+1)}_{n}-w^{(i+1)}_{n-1}\|\nonumber\\&\leq& (1-\alpha_{n})\|x_{n}-x_{n-1}\|+\theta_{n-1}W_{5}+|\alpha_{n}-\alpha_{n-1}|W_{1}\nonumber\\&&+|\lambda_{n-1}-\lambda_{n}|W_{2}+\lambda_{n-1}\beta_{n-1}W_{3}+\frac{|\lambda_{n}-\lambda_{n-1}|}{\lambda_{n-1}}W_{4},\nonumber
\end{eqnarray}
and by a similar argument as in Lemma \ref{L2} $(a)$ this yields
\begin{eqnarray}  \label{eq:bb10}
\frac{\|x_{n+1}-x_{n}\|}{\lambda_{n}}\leq (1-\alpha_{n})\frac{\|x_{n}-x_{n-1}\|}{\lambda_{n-1}}+\alpha_{n}\Lambda_{n},
\end{eqnarray}
where $
\Lambda_{n}\nonumber=\frac{\theta_{n-1}}{\lambda_{n}\alpha_{n}}W_{5}+\frac{1}{\lambda_{n}}\Big|1-\frac{\alpha_{n-1}}{\alpha_{n}}\Big| W_{1}+\sigma\frac{\beta_{n-1}}{\alpha_{n}}W_{3}+\frac{1}{\alpha_{n}}\Big|\frac{1}{\lambda_{n}}-\frac{1}{\lambda_{n-1}}\Big|W_{7},
$	
and $W_{7}=2LW_{2}+W_{4}$. Therefore, using Assumption \ref{assumption1} and Proposition \ref{P5}, it follows from (\ref{eq:bb10}) that	
$
\lim_{n\rightarrow\infty}\frac{\|x_{n+1}-x_{n}\|}{\lambda_{n}}=\lim_{n\rightarrow\infty}\|x_{n+1}-x_{n}\|=0.$
\item{$(b)$.} By $y^{(i)}_{n}=\mbox{prox}_{\lambda_{n}f_{i}}(z_{n}+\lambda_{n}d^{(i)}_{n+1})$, we get
$
\langle   z_{n}+\lambda_{n}d^{(i)}_{n+1}-y^{(i)}_{n},\bar{x}-y^{(i)}_{n}\rangle \leq \lambda_{n}(f_{i}(\bar{x})-f_{i}(y^{(i)}_{n})),$
from which it simply follows that
\begin{eqnarray}  \label{eq:bb14}
\|y^{(i)}_{n}-\bar{x}\|^{2}&\leq&\|z_{n}-\bar{x}\|^{2}-\|z_{n}-y^{(i)}_{n}\|^{2}+2\lambda_{n}(f_{i}(\bar{x})-f_{i}(y^{(i)}_{n}))\\&&\hspace{30mm}+2\lambda_{n}\langle  d^{(i)}_{n+1},y^{(i)}_{n}-\bar{x}\rangle.\nonumber
\end{eqnarray}
Furthermore, the definitions $z_{n}$ yields $\|z_{n}-\bar{x}\|^{2}\leq\|x_{n}-\bar{x}\|^{2}+\theta_{n}D_{1}$, where $D_{1}=\max_{i\in I}(\sup\{\|x_{n}-x_{n-1}\|^{2}+2\|x_{n}-x_{n-1}\|\|x_{n}-\bar{x}\|:n\in\mathbb{N}\})<\infty$, and the definition of $w^{(i+1)}_{n}$ and the firmly nonexpansiveness of the mapping $T_{i}$ gives
$\|w^{(i+1)}_{n}-\bar{x}\|^{2}\leq(1-\alpha_{n})\|T_{i}(y^{(i)}_{n})-\bar{x}\|^{2}+\alpha_{n}\|u^{(i)}-\bar{x}\|^{2}\leq\|y^{(i)}_{n}-\bar{x}\|^{2}-(1-\alpha_{n})\|T_{i}(y^{(i)}_{n})-y^{(i)}_{n}\|^{2}+\alpha_{n}D_{2},$
where $D_{2}=\max_{i\in I}(\sup\{\|u^{(i)}-\bar{x}\|^{2}:n\in\mathbb{N}\})<\infty$.  Therefore, combined with (\ref{eq:bb14}) we obtain 
\begin{eqnarray}  
\|w^{(i+1)}_{n}-\bar{x}\|^{2}\nonumber&\leq&\|x_{n}-\bar{x}\|^{2}+\theta_{n}D_{1}+2\lambda_{n}(f_{i}(\bar{x})-f_{i}(y^{(i)}_{n}))+2\lambda_{n}\langle  d^{(i)}_{n+1},y^{(i)}_{n}-\bar{x}\rangle\nonumber\\&&+\alpha_{n}D_{2}-(1-\alpha_{n})\|T_{i}(y^{(i)}_{n})-y^{(i)}_{n}\|^{2}-\|z_{n}-y^{(i)}_{n}\|^{2}\nonumber
\end{eqnarray}
for each $i\in I$.
\end{proof}
\subsection*{APPENDIX B.3. Proof of Lemma~\ref{L6}}
\begin{proof}
From the given assumptions on $f_{i}$, there exists $q^{(i)}\in\partial f_{i}(\bar{x})$ and a non-negative real number $D_{3}$ such that
$2(f_{i}(\bar{x})-f_{i}(y^{(i)}_{n}))\leq 2\|\bar{x}-f_{i}(y^{(i)}_{n})\|\|q^{(i)}\|\leq D_{3}.$ Hence, we have from Lemma \ref{L5} $(b)$ that for every $i\in I$
\begin{eqnarray}\label{eq:bb177}
\|w^{(i+1)}_{n}-\bar{x}\|^{2}\nonumber&\leq&\|x_{n}-\bar{x}\|^{2}+\theta_{n}D_{1}+\alpha_{n}D_{2}+\lambda_{n}(D_{3}+D_{4})\nonumber\\&&-(1-\alpha_{n})\|T_{i}(y^{(i)}_{n})-y^{(i)}_{n}\|^{2}-\|z_{n}-y^{(i)}_{n}\|^{2},
\end{eqnarray}
where $D_{4}=\max_{i\in I}(\sup\{2\langle  d^{(i)}_{n+1},y^{(i)}_{n}-\bar{x}\rangle:n\in\mathbb{N}\})<\infty$. Thus, (\ref{eq:bb177}) gives
\begin{eqnarray}
\|x_{n+1}-\bar{x}\|^{2}\nonumber&=&\Big\|\frac{1}{M}\sum\limits_{i=1}^{M} w^{(i+1)}_{n}-\frac{1}{M}\sum\limits_{i=1}^{M}\bar{x}\Big\|^{2}\leq\frac{1}{M}\sum\limits_{i=1}^{M}\|w^{(i+1)}_{n}-\bar{x}\|^{2}\nonumber\\&\leq&\|x_{n}-\bar{x}\|^{2}+\theta_{n}D_{1}+\alpha_{n}D_{2}+\lambda_{n}D_{5}\nonumber\\&&-(1-\alpha_{n})\frac{1}{M}\sum\limits_{i=1}^{M}\|T_{i}(y^{(i)}_{n})-y^{(i)}_{n}\|^{2}-\frac{1}{M}\sum\limits_{i=1}^{M}\|z_{n}-y^{(i)}_{n}\|^{2},\nonumber
\end{eqnarray}
where $D_{5}=D_{3}+D_{4}$. This leads to
$(1-\alpha_{n})\frac{1}{M}\sum_{i=1}^{M}\|T_{i}(y^{(i)}_{n})-y^{(i)}_{n}\|^{2}+\frac{1}{M}\sum_{i=1}^{M}\|z_{n}-y^{(i)}_{n}\|^{2}\leq\|x_{n}-x_{n+1}\|D_{6}+\theta_{n}D_{1}+\alpha_{n}D_{2}+\lambda_{n}D_{5},$
where  
$D_{6}=\sup\{\|x_{n}-\bar{x}\|+\|x_{n+1}-\bar{x}\|:n\in\mathbb{N}\}<\infty$, which implies in view of Lemma \ref{L5} $(a)$, Condition \ref{condition1} and
$\|T_{i}(y^{(i)}_{n})-z_{n}\|\leq\|T_{i}(y^{(i)}_{n})-y^{(i)}_{n}\|+\|z_{n}-y^{(i)}_{n}\|$ that
\begin{eqnarray}\label{eq:bb20}
\lim\limits_{n\rightarrow\infty}\|T_{i}(y^{(i)}_{n})-y^{(i)}_{n}\|=\lim\limits_{n\rightarrow\infty}\|z_{n}-y^{(i)}_{n}\|=\lim\limits_{n\rightarrow\infty}\|T_{i}(y^{(i)}_{n})-z_{n}\|=0.
\end{eqnarray}
The definition of $w_{n}^{(i+1)}$ and $z_{n}$ gives  $\|w_{n}^{(i+1)}-T_{i}(y^{(i)}_{n})\|=\alpha_{n}\|u^{(i)}-T_{i}(y^{(i)}_{n})\|$  and  $\|x_{n}-z_{n}\|=\theta_{n}\|x_{n}-x_{n-1}\|.$ Since the sequences $\{u^{(i)}-T_{i}(y^{(i)}_{n})\}$ ($i\in I$) and $\{x_{n}-x_{n-1}\}$ are bounded and $ \lim_{n\rightarrow\infty}\alpha_{n}=\lim_{n\rightarrow\infty}\theta_{n}= 0$, we have
$
\lim_{n\rightarrow\infty}\|w_{n}^{(i+1)}-T_{i}(y^{(i)}_{n})\|= \lim_{n\rightarrow\infty}\|x_{n}-z_{n}\|=0$, and so it follows together with (\ref{eq:bb20}) that
\begin{eqnarray}\label{eq:bb22}
\lim\limits_{n\rightarrow\infty}\|w^{(i+1)}_{n}-y_{n}^{(i)}\|=\lim\limits_{n\rightarrow\infty}\|w^{(i+1)}_{n}-z_{n}\|=\lim\limits_{n\rightarrow\infty}\|x_{n}-y^{(i)}_{n}\|=0.
\end{eqnarray}
Therefore, the inequality    
$\|x_{n}-T_{i}(x_{n})\|=\|x_{n}-y^{(i)}_{n}\|+\|y^{(i)}_{n}-T_{i}(y^{(i)}_{n})\|+\|T_{i}(x_{n})-T_{i}(y^{(i)}_{n})\|\leq2\|x_{n}-y^{(i)}_{n}\|+\|y^{(i)}_{n}-T_{i}(y^{(i)}_{n})\|$ together with (\ref{eq:bb20}) and (\ref{eq:bb22}) gives	
$\lim_{n\rightarrow\infty}\|x_{n}-T_{i}(x_{n})\|=0
$ for each $i\in I$.
\end{proof}


\end{document}